\documentclass[12pt]{amsart}
\usepackage{color,xcolor,comment,bookmark}
\usepackage{enumerate}
\usepackage{enumitem}

\allowdisplaybreaks

\usepackage[english]{babel}
\usepackage{amsmath}
\usepackage{amsthm}
\usepackage{amsfonts}
\usepackage{amssymb}
\usepackage{hyperref}
\usepackage{appendix}
\usepackage{mathtools}
\usepackage{graphicx}

\newtheorem{theorem}{Theorem}
\newtheorem{proposition}{Proposition}

\newtheorem{remark}{Remark}

\newtheorem{lemma}{Lemma}

\newcommand{\e}{\varepsilon}

\newcommand{\Jrho}{\mathcal{J_\varepsilon}\rho^{\varepsilon,\delta}}

\newcommand{\Jrhouno}{\mathcal{J_\varepsilon} \partial_x\rho^{\varepsilon,\delta}}
\newcommand{\J}{\mathcal{J}}

\newcommand{\JA}{\mathcal{J_\varepsilon}A^{\varepsilon,\delta}}
\newcommand{\JAuno}{\mathcal{J_\varepsilon}\partial_x A^{\varepsilon,\delta}}

\newcommand{\rhoed}{\rho^{\varepsilon,\delta}} 
\newcommand{\Aed}{A^{\varepsilon,\delta}} 

\newcommand{\D}{\,d}

\usepackage{xcolor}

\title[Pressure-Driven Cross-Diffusive Porous-Medium System]{Analysis of a Cross-Nonlinear Porous-Medium System Modeling Pressure-Driven Cell Population Dynamics}

\author[A. B\'ejar-L\'opez]{Alexis B\'ejar-L\'opez}
\address[Alexis B\'ejar-L\'opez]{\newline Departamento de Matem\'atica Aplicada and Research Unit ``Modeling Nature'' (MNat), Facultad de Ciencias, Universidad de Granada, 18071 Granada, Spain}
\email{alexisbejar@ugr.es}

\author[R. Granero-Belinch\'on]{Rafael Granero-Belinch\'on}
\address[Rafael Granero-Belinch\'on]{\newline Departamento  de  Matem\'aticas,  Estad\'istica  y  Computaci\'on,  Universidad  de Cantabria.  Avda.  Los  Castros  s/n,  Santander,  Spain.}
\email{rafael.granero@unican.es}

\author[C. Pulido]{Carlos Pulido}
\address[Carlos Pulido]{\newline Departamento de Matem\'atica Aplicada and Research Unit ``Modeling Nature'' (MNat), Facultad de Ciencias, Universidad de Granada, 18071 Granada, Spain}
\email{cpulidog@ugr.es}

\author[J. Soler]{Juan Soler}
\address[Juan Soler]{\newline Departamento de Matem\'atica Aplicada and Research Unit ``Modeling Nature'' (MNat), Facultad de Ciencias, Universidad de Granada, 18071 Granada, Spain}
\email{Jsoler@ugr.es}

\begin{document}

\overfullrule 1mm



\thanks{\textbf{Acknowledgment.}  This work has been partially supported by the grants: DMS Grant  1908739, 2049020, 2205694 and 2219397 by the NSF (USA);  by the State Research Agency of the Spanish Ministry of Science and FEDER-EU, project PID2022-137228OB-I00 (MICIU/AEI /10.13039/501100011033); by Modeling Nature Research Unit, Grant QUAL21-011 funded by Consejer\'ia de Universidad, Investigaci\'on e Innovaci\'on (Junta de Andaluc\'ia); by Grant C-EXP-265-UGR23 funded by
Consejer\'ia de Universidad, Investigaci\'on e Innovaci\'on \& ERDF/EU
Andalusia Program, by grant PID2022-141187NB-I00 (MCIN/AEI/10.13039/501100011033/FEDER, UE) and by the Spanish Ministry of Science, Innovation and Universities FPU research grant FPU19/01702 (A. B-L). R.G.B thanks the hospitality of the MNat research unit during his research stay at Universidad de Granada during part of this research was carried out.}

	\begin{abstract}
		In this work, we introduce a cross-diffusion model that couples population density and occupied area to investigate how internal pressure drives growth and motility. By blending nonlinear nonlocal interactions with porous-medium diffusion and an antidiffusive pressure term, the model captures the two-way feedback between local density fluctuations and tissue expansion or contraction. Building on Shraiman’s area-growth paradigm, we enrich the framework with density-dependent spreading at the population boundary and a novel cross-diffusion term, yielding fully nonlinear transport in both equations. We prove local well-posedness for nonnegative solutions in Sobolev spaces and, under higher regularity, show both density and area remain nonnegative. Uniqueness follows when the initial density’s square root lies in $H^2$, even if density vanishes on parts of the domain. We also exhibit initial data that induce finite-time blow-up, highlighting potential singularity formation. Finally, we establish that the density’s spatial support remains invariant and characterize the co-evolution of occupied area and population density domains, offering new insights into pattern formation and mass transport in biological tissues. 
	\end{abstract}
 
	\maketitle
    \tableofcontents
    
    \section{Introduction and Statement of Main Results}
	\label{sec:1}
	
The objective of this paper is to analyze the influence of pressure on a growing, motile population by formulating a cross-diffusion model that couples the evolution of the population density with its occupied area. The resulting system of partial differential equations features competing nonlinear and nonlocal interactions, a porous-medium–type diffusion term, and an antidiffusive (aggregation or contractive) contribution arising from the system’s internal pressure.

The resulting model can therefore be written in terms of the pair $(A,\rho)$, denoting the occupied area and the population density respectively, and takes the form of the following coupled evolution equations:
	  \begin{equation}\label{Original}		
		\begin{aligned}
			\partial_t A&= A(\alpha\rho-\mu\alpha(\rho-\langle \rho \rangle)+\widetilde{\beta}A\left(1-\frac{\rho A}{\widetilde{K}} \right)+\partial_x(\rho \partial_x A) ,\\
			\partial_t\rho&=\beta\ \rho\left(1-\frac{A\rho}{K} \right)-\alpha\rho^2+\mu\alpha\rho(\rho-\langle \rho\rangle)+\partial_x(\rho \partial_{x}\rho),
		\end{aligned}
	\end{equation}
    where, for any even kernel $\Gamma\in L^1(\mathbb{R})$, we write
\[
\langle f\rangle \;=\;\Gamma * f.
\]
This system is subject to the initial data
\[
			 A(0,x)  =A_0,\;\rho(0,x)=\rho_0,
\]
and periodic boundary conditions
\[
			 A(t,-L) =A(t,L),\;\rho(t,-L)=\rho(t,L).
\]
This system builds on the seminal model of Shraiman \cite{Shraiman2005}, in which the internal pressure of a proliferating cell population drives the expansion of its occupied area, and that area growth in turn influences local cell density.  In Shraiman's original framework, cell motility was not explicitly modeled; the dynamics followed only the temporal evolution of density and area.

In \cite{Blanco2021} was performed an extensive numerical study, enriching the model with diverse transport mechanisms. By introducing a porous--medium or flux--saturated term, they achieved finite‐speed propagation of cell fronts--despite the presence of the anti--diffusive term $ \alpha(\rho - \langle\rho\rangle)$ and the fact that, in their formulation, the parameter $\alpha$ itself was a pressure--modulating function. These mobility laws were shown to be decisive in shaping both pattern formation and front dynamics.

In the present work, we further advance Shraiman's area--growth paradigm by incorporating density--dependent spreading at the population's support boundary.  To this end, we endow the density equation with a porous--medium--type diffusion--thereby ensuring finite--speed, concentration--sensitive propagation--and, in recognition of the intrinsic coupling between area and density, introduce a corresponding cross--diffusion term in the area equation.  This novel coupling faithfully captures how local variations in density drive both the expansion and contraction of the tissue's effective area.

Cross‐diffusion systems have become a versatile framework for capturing complex transport and interaction mechanisms in both biological and social contexts. Notable examples include pathogen–chemotaxis models of viral spread \cite{Stancevic2013,Painter2019}, reaction–diffusion descriptions of criminal hot‐spot formation \cite{Short2008,Rodriguez2020}, and multi‐species competition for spatially distributed resources whose availability depends on local population densities \cite{Tsyganov2003,Tania2012,Tello2016,Cintra2018,Tao2019}. See \cite{bellomo2022chemotaxis} for an extensive review of cross-diffusion applications in chemotaxis. In the vast majority of these models, the nonlinear cross‐diffusion term enriches only one equation, while the remaining equations retain standard (linear) diffusion. In contrast, the system we propose here features fully nonlinear transport in both components: one equation includes a self‐diffusion porous‐medium–type, and the other a genuine nonlinear cross‐diffusion coupling, leading to a richer interplay between intra‐ and inter‐population dynamics.

Regarding the analytic theory of cross‐diffusion systems, we briefly highlight several foundational contributions. In \cite{Alasio2022} was established the existence and Sobolev‐regularity of solutions to a class of nonlinear, degenerate parabolic systems featuring both self‐ and cross‐diffusion alongside nonlocal interaction terms.  \cite{Mielke2023} constructed weak and very‐weak solutions for two‐component coupled degenerate parabolic PDEs.  \cite{Winkler2025} proved global‐in‐time existence of continuous weak solutions—under suitable Sobolev regularity of the initial data—for models that combine nonlinear cross‐diffusion with classical linear diffusion. Although the literature on porous medium type equations is large, the works dealing with porous medium type systems are much more scarce. In this regards, the interested reader could also refer to the works \cite{bulivcek2019large,cuvillier2023well,fanelli2022finite,Fanelli2023, kosewski2025local,mielke2023two,mielke2022existence} where a system of porous medium type arising in turbulence is studied. Finally,  \cite{chen2018global} provided a rigorous analysis of multi‐species cross‐diffusion models, deriving global‐in‐time existence results via a refined entropy framework that accommodates the nonlinear coupling. 

When multiple populations share the same environment, a distinct form of pressure—known as homeostatic pressure—arises. In \cite{basan2009homeostatic, ranft2014mechanically} a model based on this pressure is introduced to study the interface between populations. This nonlinear transport framework prescribes a propagation velocity that, via an associated Helmholtz equation, depends on a nonlinear combination of higher-order derivatives of the density. A rigorous analysis of the resulting traveling-wave solutions is carried out in \cite{campos2025biomechanical}.
\subsubsection*{\textbf{Local--in--Time Existence of Solution}}
    The primary goal of this paper is to establish the well‐posedness of system \eqref{Original} for regular solutions in the Sobolev‐space framework.  In particular, we establish local existence of solutions 
\[
(A,\rho)\;\in\;C\bigl([0,T);H^{m-1-s}(I)\bigr)\times C\bigl([0,T);H^{m-s}(I)\bigr)
\]
for every integer $m\ge3$, for every $0<s\ll1$, and any non‐negative initial data 
$\,(A_0,\rho_0)\in H^{m-1}(I)\times H^m(I)$, where $I=[-L,L]$.  
Given the biological interpretation of $\rho$, we restrict our attention to solutions that remain non‐negative for all $t\in[0,T)$.  Under the additional regularity hypothesis $m\ge4$, we prove the existence of such non‐negative solutions.  This statement is the content of Theorem \ref{TheoremExistence}, which we now formulate:

\begin{theorem}[Local existence and nonnegativity]\label{TheoremExistence}
Let $m\ge3$ be an integer, and let
\[
(A_0,\rho_0)\in H^{m-1}(I)\times H^m(I)
\]
be a pair of nonnegative initial data.  Then there exists a maximal time $T>0$ and a solution $(A,\rho)$  to system \eqref{Original} such that
\[
(A,\rho)\in C\bigl([0,T);H^{m-1 -s}(I)\bigr)\times C\bigl([0,T);H^{m -s}(I)\bigr), \quad 0<s\ll1,
\]
and
$$(A,\rho)\in L^\infty \bigl([0,T);H^{m-1}(I)\bigr)\times L^\infty\bigl([0,T);H^m (I)\bigr),$$  satisfying $\rho(t,x)\ge0$, for all $(t,x)\in[0,T)\times I$.  Moreover, if $m\ge4$, then in fact
\[
A(t,x)\ge0
\quad
\text{for all }(t,x)\in[0,T)\times I.
\]
\end{theorem}
\begin{remark}
    Note, by combining the system’s equations with the established regularity of the solutions, we also deduce that
    \[
\left(\frac{\partial A}{\partial t},\frac{\partial\rho} {\partial t}\right)\in C\bigl((0,T);H^{m-3-s}(I)\bigr)\times C\bigl((0,T);H^{m-2-s}(I)\bigr)
\]
\end{remark}

The proof of Theorem \ref{TheoremExistence} proceeds via a two‐parameter regularization and subsequent limit passage. We introduce a family of approximate problems, indexed by a smoothing parameter $\varepsilon>0$ and a positivity‐enforcing parameter $\delta>0$, that converge to system \eqref{Original}. The $\varepsilon$–regularization mollifies the equations, while adding $\delta$ to the initial data guarantees strict positivity. This positivity activates the coercivity of the parabolic term, from which we derive uniform energy estimates independent of $\varepsilon$ and $\delta$. These bounds allow us to extract a convergent subsequence and pass to the limit by compactness arguments. See Section~\ref{sec:2} for further details. Importantly, this method underlines the necessity of preserving strict positivity to control the degenerate nonlinearities and ensures that the resulting well‐posedness theory aligns with the biological constraint $\rho\ge0$.  The strong regularity we require and propagate in time is crucial for applying pointwise estimates in the demonstration of finite‐time singularity formation (cf. Theorem 3).  Moreover, one of
our main conclusions is that any alteration of the solution’s support
can occur only through the development of a singularity.

\subsubsection*{\textbf{Uniqueness in the Local Setting}}
The following result shows that the solution from Theorem~\ref{TheoremExistence} is unique once we impose the additional regularity
\[
\sqrt{\rho_{0}}\in H^{2}(I).
\]

\begin{theorem}[Uniqueness]\label{TheoremUniqueness}
Under the assumptions of Theorem~\ref{TheoremExistence}, if furthermore
\[
\sqrt{\rho_{0}}\in H^{2}(I),
\]
then  there exists a maximal time $T>0$ such that the solution satisfies
\[
\sqrt{\rho}\in C \bigl([0,T);H^{2-s}(I)\bigr) \cap L^\infty \bigl([0,T);H^{2}(I)\bigr),
\]
and the pair $(A,\rho)$ is uniquely determined.
\end{theorem}

The proof hinges on an $L^2$‐stability estimate for the system satisfied by $(A,\sqrt{\rho})$. Note that our assumption permits $\rho$ to vanish, since we only require $\sqrt{\rho}\in H^2$ and do not assume strict positivity. 

To complete our well‐posedness analysis of \eqref{Original}, we show that the local solution provided by Theorem \ref{TheoremExistence} need not extend globally in time.  More precisely, one can exhibit initial data for which finite‐time blow‐up occurs:

\subsubsection*{\textbf{Non-Global Existence in Time}} Regarding the existence of finite‐time singularities, we can stablish that
there exist nonnegative initial data 
\[
(A_0,\rho_0)\in H^{m-1}(I)\times H^m(I),
\qquad m\ge4,
\]
such that the corresponding solution from Theorem \ref{TheoremExistence} becomes singular in finite time.
The proof of this result, which uses pointwise methods, proceeds by selecting $\rho_0$ such  that it vanishes sharply at some point and whose second derivative there exceeds a critical threshold.  Assuming no other singularity precedes it, one then shows that $\partial_x^2\rho(t,0)\to+\infty$ as $t$ approaches a finite blow‐up time.  The precise conditions on $\rho_0$ are given in Theorem \ref{TheoremBlowUp}. See also Fig.~\ref{fig-1}, which illustrates the results of this theorem.

\subsubsection*{\textbf{Dynamics of the Spatial Support}}
Next, we examine how the spatial supports of $\rho$ and $A$ evolve over time. In particular, we determine whether their positivity sets can spread beyond the initial domain or remain confined—a question that is pivotal for understanding pattern formation, mass transport, and the effects of nonlocal interactions.

The following theorem asserts that the support of 
$\rho$ remains invariant over the existence interval on which the solution’s regularity is preserved. We also show that support of $A$ reaches the support of $\rho$ for any positive time under certain hypothesis on initial data.

\begin{theorem}[Dynamics of support]\label{TheoremSupport}
Suppose
\[
(A_0,\rho_0)\,\in\,H^3(I)\times H^4(I),
\qquad
\sqrt{\rho_0}\,\in\,H^2(I),
\]
and let $T>0$ be such that the statements of Theorems \ref{TheoremExistence} and \ref{TheoremUniqueness} are satisfied.
Then for every $t\in[0,T)$,
\[
\operatorname{supp}\rho(t,\cdot)
\;=\;
\operatorname{supp}\rho_0.
\]
Moreover, if $\rho_0$ is supported on an interval and 
$\displaystyle\operatorname{supp}A_0\subset\operatorname{supp}\rho_0$,
then one proves that 
\[
\operatorname{supp}A(t)
=\operatorname{supp}\rho(t)
\quad\forall\,t\in(0,T).
\]
\end{theorem}

A pivotal element in the proof is to recast the $\rho$–equation in its mild (integral) form. By exploiting the $H^2$–regularity of $\sqrt{\rho}$ together with the continuity of $A$ and $\rho$, one shows that any zero of $\rho_0$ is preserved for all $t\in[0,T)$, so that 
\[
\operatorname{supp}\rho(t)\;\subset\;\operatorname{supp}\rho_0.
\]

This invariance encapsulates the natural (biological) constraint that the density cannot occupy regions devoid of area and underscores the intrinsic spatial coupling between $A$ and $\rho$. 
On the other hand, the proof of the evolution of the support of $A$ hinges on two principal ideas. First, in any region where $A_0>0$, the strict positivity of $A$ is preserved for all $t\in[0,T)$. Second, on each compact subset where $\rho$ remains uniformly bounded away from zero, the equation for $A$ becomes uniformly parabolic. An application of the parabolic maximum principle then shows that $A$ instantaneously becomes strictly positive at every point where $A_0$ originally vanished.
The complete and rigorous statements and proofs of these results are given in Proposition \ref{SupportRho} and Proposition \ref{SupportA}. See also Fig.~\ref{fig-2} that illustrates these results.

\subsection{Overview of the Model’s Foundations}

Let us now briefly outline the model derivation. Following Shraiman's approach, we regard the cell population as an elastic continuum subject to spatially heterogeneous growth. By minimizing the total elastic strain energy, one finds that the resulting deformation drives the evolution of the occupied area according to:
\[
\partial_t A = A \left[\alpha\rho-\frac{\gamma\kappa}{\gamma+\kappa}(\alpha\rho-\alpha\langle\rho\rangle)\right] .
\]
Here, $\rho$ denotes the population density. Writing 
\[
\rho = \frac{N}{A},
\]
with $N$ the total cell number, and assuming logistic proliferation of the form 
$\tfrac{dN}{dt}=\beta N\bigl(1-\tfrac{N}{K}\bigr)$,
one obtains
\[
\partial_t N = \beta_1 N\Big(1-\frac{N}{K_1}\Big)
\]
Recasting the above relation in terms of the population density yields:

$$
\partial_t \rho=\beta_1\rho \Big(1-\frac{\rho A}{K_1}\Big)-\alpha\rho^2-\alpha\frac{\gamma \kappa}{\gamma+\kappa}\rho(\alpha\rho-\alpha\langle\rho\rangle)
$$
Hence, under the assumption that spatially heterogeneous growth induces elastic deformation, Shraiman's area--evolution model can be written as
:\begin{align*}
	\partial_t A &= A \left[\alpha\rho-\frac{\gamma\kappa}{\gamma+\kappa}(\alpha\rho-\alpha\langle\rho\rangle)\right]\\
	\partial_t \rho&=\beta_1\rho \left(1-\frac{\rho A}{K_1}\right)-\alpha\rho^2+\alpha\frac{\gamma\kappa}{\gamma+\kappa}\rho(\alpha\rho-\alpha\langle\rho\rangle)
\end{align*}
The subsequent refinement in \cite{Blanco2021} involved augmenting the density equation with both advective and diffusive transport terms, followed by a systematic parameter study calibrated to biologically measured tissue mechanics.  In particular, the tissue's elastic behavior is characterized by the bulk modulus $\kappa$ and shear modulus $\gamma$, which for many cell aggregates are (see \cite{Blanco2021}) on the order of
\[
\kappa \approx 3.33\ \mathrm{kPa},
\qquad
\gamma \approx 0.34\ \mathrm{kPa}.
\]
These values indicate that volumetric deformations are nearly an order of magnitude stiffer than shear deformations, reflecting the tissue's almost incompressible response.

Accordingly, introducing the effective coupling constant
\[
\mu \;=\;\frac{\gamma\,\kappa}{\gamma + \kappa}\,\in(0,1),
\]
the area--evolution equation takes the succinct form
\[
\partial_t A = A \left[\alpha\rho-\mu(\alpha\rho-\alpha\langle\rho\rangle)\right]
\]
Note that whenever $\rho>0$, each term on the right--hand side is strictly positive, so with biologically realistic parameter values the model predicts unbounded area growth. To prevent this runaway expansion, one appends a logistic--type saturation term to the area equation:
\[
\partial_t A = A \left[\alpha\rho-\mu(\alpha\rho-\alpha\langle\rho\rangle)\right]+\widetilde{\beta}A\left(1-\frac{\rho A}{\widetilde{K}} \right) .
\]
Here, the area's carrying capacity scales inversely with the local cell density: higher $\rho$ corresponds to increased compression of cells and thus a reduced maximal area.

Consequently, the density equation is modified to:
$$
\partial_t \rho =\beta_1\rho \left(1-\frac{\rho A}{K_1}\right)-\alpha\rho^2+\alpha\mu\rho(\alpha\rho-\alpha\langle\rho\rangle)-\widetilde{\beta}\rho\left(1-\frac{\rho A}{\widetilde{K}} \right).
$$
Introducing the shifted parameters
\[
\beta = \beta_{1} - \widetilde\beta,
\qquad
K = \bigl(\tfrac1{K_{1}} - \tfrac1{\widetilde K}\bigr)^{-1},
\]
and assuming $\beta>0$ and $K>0$, the governing equations become
 \[
		\partial_t A= A(\alpha\rho-\mu\alpha(\rho-\langle \rho \rangle)+\widetilde{\beta}A\left(1-\frac{\rho A}{\widetilde{K}} \right) ,\\
\]
for the area, and
\[
		\partial_t\rho=\beta\ \rho\left(1-\frac{A\rho}{K} \right)-\alpha\rho^2+\mu\alpha\rho(\rho-\langle \rho\rangle),
\]
for the density.

Finally, extending the transport--enriched framework of \cite{Shraiman2005}, we endow the density equation with porous--medium--type diffusion and introduce a density--dependent diffusion term into the area equation, yielding system \eqref{Original}. 

Porous--medium--type diffusion captures the fact that cells migrating through a tissue--modeled as a porous medium--must first reorganize within their initial support, generating internal pressure that concentrates at the boundary. For a comprehensive introduction to porous‐medium equations, see \cite{vazquez2007porous}. This pressure accumulation produces a loss of regularity in the density profile--specifically, blow--up of the second spatial derivative--which then dynamically initiates propagation. 

As Section 5's simulations demonstrate, the density profile steepens until a curvature singularity appears at the frontier, in agreement with the finite--time blow--up described in Theorem \ref{TheoremBlowUp}. By coupling area transport directly to density flux, we ensure that the occupied area evolves in lockstep with cellular rearrangements and remains confined to the density support throughout (cf. Theorem \ref{TheoremSupport}). Crucially, the area does not evolve autonomously but is confined to the density, reflecting the fundamental mechanical interdependence of tissue deformation and cell packing.

\section{Local--in--Time Existence of Solutions}
	\label{sec:2}

As noted in the Introduction, the local--in--time existence of solutions is established by applying a compactness argument to a family of suitably regularized approximations of \eqref{Original}.  To this end, 
we denote $\J_\e$ the periodic heat kernel at time $t=\e$. We then replace each occurrence of $A$ and $\rho$ in \eqref{Original} by the convolution $\J_\e A$, $\J_\e\rho$, thereby obtaining a sequence of smooth, uniformly parabolic regularized problems.  By deriving suitable energy estimates and invoking compactness arguments, we can extract a convergent subsequence whose limit is a solution of the original system.

	Under this notation,  we introduce the following regularized problem
	\begin{equation}\label{RegularizedEpsDelta}
			\begin{split}
			 \partial_t A^{\varepsilon,\delta}=& \mathcal{J_\varepsilon}\left( \JA(\alpha\Jrho-\mu\alpha(\Jrho-\langle \Jrho \rangle))\right)\\
			&+\hspace{-0,1cm}\mathcal{J_\varepsilon}\left(\widetilde{\beta}\JA\left(1\hspace{-0,1cm}-\hspace{-0,1cm}\frac{\Jrho \JA}{\widetilde{K}} \right)\hspace{-0,1cm}+\hspace{-0,1cm} \partial_x(\Jrho \JAuno)\right),\\ 
			\partial_t\rho^{\varepsilon,\delta}=&\mathcal{J_\varepsilon}\left(\beta\ \Jrho\left(1\hspace{-0,1cm}-\hspace{-0,1cm}\frac{\JA\Jrho}{K} \right)-\alpha(\Jrho)^2\right)\\
			&+\hspace{-0,1cm}\mathcal{J_\varepsilon}\left(\mu\alpha\Jrho(\Jrho\hspace{-0,1cm}-\hspace{-0,1cm}\langle \Jrho\rangle)\hspace{-0,1cm}+\hspace{-0,1cm}\partial_{x}(\Jrho \mathcal{J_\varepsilon}\partial_{x}\rho^{\varepsilon,\delta})\right),
            \end{split}
	        \end{equation}
        together with the initial and periodicity conditions
\begin{equation}
\begin{split}
			A^{\varepsilon,\delta}(0,x)&=\mathcal{J}_\varepsilon(\delta+A_0),\quad \rho^{\varepsilon,\delta}(0,x)=\mathcal{J}_\varepsilon(\delta+\rho_0),\\
			A^{\varepsilon,\delta}(t,-L)&=A^{\varepsilon,\delta}(t,L),\quad\rho^{\varepsilon,\delta}(t,-L)=\rho^{\varepsilon,\delta}(t,L).			 
		\end{split}
	\end{equation}
	
	    The two regularization parameters $\varepsilon$ and $\delta$ play complementary roles in our approach: $\varepsilon$ controls the mollification of the system, while $\delta$ enforces a uniform lower bound on $\Jrho$ throughout the evolution. This strict positivity is indispensable for deriving \emph{a priori} $H^m$–bounds on the approximate densities $\rho^{\varepsilon,\delta}$, since it permits us to harness the parabolic dissipation
\[
- \int_I \Jrho \,\bigl(\partial_x^{m+1}\Jrho\bigr)^2 \,\mathrm{d}x
\]
to absorb the nonlinear terms generated by the logistic interaction, for example,
\[
\int_I \partial_x^m \JA \,\bigl(\Jrho\bigr)^2\,\partial_x^m\Jrho \,\mathrm{d}x.
\]
Moreover, one shows that both $\rho^{\varepsilon,\delta}$ and $A^{\varepsilon,\delta}$ remain uniformly bounded away from zero—namely,
\[
\bigl(\rho^{\varepsilon,\delta}\bigr)^{-1},\ \bigl(A^{\varepsilon,\delta}\bigr)^{-1}\;\in\;L^\infty(I)
\]
independently of $\varepsilon$.
	In Lemma \ref{LemmaExistenceEpsilon} we establish the existence of a solution to Problem \eqref{RegularizedEpsDelta}. Then, in Lemmas \ref{LemmaExistenceDelta1} and \ref{LemmaExistenceDelta2}, we let $\varepsilon\to0$ to demonstrate that the corresponding sequence of solutions 
 to the problem
	\begin{equation}\label{RegularizedDelta}
			\begin{aligned}
			\partial_t A^\delta=&  A^\delta(\alpha\rho^\delta-\mu\alpha(\rho^\delta\hspace{-0,1cm}-\hspace{-0,1cm}\langle \rho^\delta \rangle))+\widetilde{\beta}A^\delta\left(1\hspace{-0,1cm}-\hspace{-0,1cm}\frac{\rho^\delta A^\delta}{\widetilde{K}} \right)+\partial_x(\rho^\delta \partial_x A^\delta), \\
			\partial_t\rho^\delta=&\beta\ \rho^\delta\left(1\hspace{-0,1cm}-\hspace{-0,1cm}\frac{A^\delta\rho^\delta}{K} \right)-\alpha(\rho^\delta)^2
			\hspace{-0,1cm}+\hspace{-0,1cm}\mu\alpha\rho^\delta(\rho^\delta-\langle \rho^\delta\rangle)+\partial_x(\rho^\delta \partial_{x}\rho^\delta),
            \end{aligned}
	\end{equation}
    together with boundary conditions
    \begin{equation}
\begin{split}
			A^\delta(0,x)&=\delta+A_0,\quad\rho^\delta(0,x)=\delta+\rho_0,\\ A^\delta(t,-L)&=A^\delta(t,L),\quad\rho^\delta(t,-L)=\rho^\delta(t,L) ,
		\end{split}
	\end{equation}
	is well defined.
	
To finish the proof of Theorem \ref{TheoremExistence}, we then pass to the limit $\delta \to 0$ in the intermediate, $\delta$\nobreakdash–regularized problem \eqref{RegularizedDelta}, thereby recovering a solution of the original system.

The proof of Lemma \ref{LemmaExistenceEpsilon} is a straightforward application of the Picard--Lindel\"of theorem in a Banach space setting (see \cite{Hartman2002}).

	\begin{lemma}\label{LemmaExistenceEpsilon}
	Let $m\ge3$ be an integer and let the initial data satisfy
$$
A_0\in H^{m-1}(I),\quad \rho_0\in H^m(I), 
\qquad A_0\ge0,\ \rho_0>0.
$$

Then for each pair of regularization parameters $\varepsilon>0$, $\delta>0$, there exists a time
$$
T_{\varepsilon,\delta}>0
$$
and a unique classical solution
$$
\bigl(A^{\varepsilon,\delta}(t,\cdot),\;\rho^{\varepsilon,\delta}(t,\cdot)\bigr)
\;\in\;
C^1\bigl([0,T_{\varepsilon,\delta}),\,H^{m-1}(I)\bigr)
\times
C^1\bigl([0,T_{\varepsilon,\delta}),\,H^m(I)\bigr)
$$
of the regularized problem \eqref{RegularizedEpsDelta},
with
$$
A^{\varepsilon,\delta}(t,x)>0,\quad \rho^{\varepsilon,\delta}(t,x)>0
\quad\text{for all }t\in[0,T_{\varepsilon,\delta}),\;x\in I.
$$        
	\end{lemma}
	\begin{proof}
		Given $\varepsilon,\delta>0$,  we start rewriting the Problem  \eqref{RegularizedEpsDelta} as
		\begin{equation*}
			\partial_t (A^{\varepsilon,\delta},\rho^{\varepsilon,\delta})=F(A^{\varepsilon,\delta},\rho^{\varepsilon,\delta}),
		\end{equation*}
		where $F:H^{m-1}(I)\times H^m(I)\to H^{m-1}(I)\times H^{m}(I)$ denotes the operator given by the right-hand side of the system.
		
		As noted above, the existence and uniqueness of a solution
$$
(A^{\varepsilon,\delta},\,\rho^{\varepsilon,\delta}) 
\in C^1\bigl([0,T_{\varepsilon,\delta}),H^{m-1}(I)\bigr)\times
C^1\bigl([0,T_{\varepsilon,\delta}),H^{m}(I)\bigr)
$$
is guaranteed on some time interval $[0,T_{\varepsilon,\delta})$ by the Picard--Lindelöf theorem.  Due to the presence of the convolutions with the heat kernel, the hypotheses of the Picard--Lindelöf theorem are easy to verify, for the sake of brevity, we left for the interested reader.

		Furthermore, since $(A^{\varepsilon,\delta},\rho^{\varepsilon,\delta})$ depends continuously on time and starts from strictly positive initial data, it stays strictly positive on a (possibly smaller) interval $[0,T_{\varepsilon,\delta})$.
	\end{proof}
	The next step is to establish the existence of a solution to Problem \ref{RegularizedDelta} by letting $\varepsilon\to0$ in the family $\bigl(A^{\varepsilon,\delta},\rho^{\varepsilon,\delta}\bigr)$, via Aubin--Lions compactness Lemma (\cite[Corrollary 4]{Simon1986})

	In order to apply this result let us show that the energy
	\[
        \begin{split}
            E(t)=& 1+\|\rho^{\varepsilon,\delta}\|_{L^2}^2+\|\partial_{x}^{m}\rho^{\varepsilon,\delta}\|_{L^2}^2+\left\| \frac{1}{\rho^{\varepsilon,\delta}}\right\|_{L^\infty} \\&+ \|A^{\varepsilon,\delta}\|_{L^2}^2+\|\partial_x^{m-1}A^{\varepsilon,\delta}\|_{L^2}^2+\left\| \frac{1}{A^{\varepsilon,\delta}}\right\|_{L^\infty}.
        \end{split}
        \]
	is uniformly bounded in some time interval $[0,T_\delta)$ where $T_\delta$ does not depend on $\varepsilon$. In particular, since on the bounded interval $I$ the Sobolev norm $\|f\|_{H^m}$ is equivalent to 
\[
\|f\|_{L^2} + \bigl\|\partial_x^m f\bigr\|_{L^2},
\]
it follows that, for each fixed $\delta>0$, the approximate solutions 
\[
\bigl(A^{\varepsilon,\delta},\,\rho^{\varepsilon,\delta}\bigr)_{\varepsilon>0}
\]
are uniformly bounded (independently of $\varepsilon$) in 
\[
C\bigl([0,T_\delta);H^{m-1}(I)\bigr)\,\times\,C\bigl([0,T_\delta);H^m(I)\bigr).
\] 
Furthermore, the  control of the terms $\left\| \frac{1}{A^{\varepsilon,\delta}}\right\|_{L^\infty}$ and $\left\| \frac{1}{\rho^{\varepsilon,\delta}}\right\|_{L^\infty}$, which appear into  our energy, guarantees that $A^{\varepsilon,\delta}(t,\cdot)$ and $\rho^{\varepsilon,\delta}(t,\cdot)$  are positive in $[0,T_\delta)$. 

In the next result we obtain the uniform bounds on the energy $E$ which allow us to apply Aubin--Lions lemma and ensure the existence of limit for the family $(A^{\varepsilon,\delta},\rho^{\varepsilon,\delta})$ as $\varepsilon$ tends to $0$.

\begin{lemma} \label{LemmaExistenceDelta1}
		Given any $\delta>0$ there exists a time $T_{\delta}>0$, a subsequence (still denoted $(A^{\varepsilon,\delta},\rho^{\varepsilon,\delta})$), and a pair of limit functions
$$
\bigl(A^\delta,\rho^\delta\bigr)\;\in\;C\bigl([0,T_\delta),H^{m-1-s}(I)\bigr)\;\times\;C\bigl([0,T_\delta),H^{m-s}(I)\bigr),
$$
such that, as $\varepsilon\to0$,
\[
\begin{split}
& A^{\varepsilon,\delta}\;\longrightarrow\;A^\delta
\quad\text{in }C\bigl([0,T_\delta),\,H^{m-1-s}(I)\bigr),
\\
&\rho^{\varepsilon,\delta}\;\longrightarrow\;\rho^\delta
\quad \ \, \text{in }C\bigl([0,T_\delta),\,H^{m-s}(I)\bigr).
\end{split}
\]
Moreover,  the limits $A^\delta(t,\cdot)$ and $\rho^\delta(t,\cdot)$ remain strictly positive for all $t\in[0,T_\delta)$.
	\end{lemma}
		\begin{proof}
        The first step is to show that there is a time $T_\delta>0$--independent of $\varepsilon$--such that the family $\bigl(A^{\varepsilon,\delta},\rho^{\varepsilon,\delta}\bigr)$ remains strictly positive and uniformly bounded in
$$
C\bigl([0,T_\delta),H^{m-1}(I)\bigr)\times C\bigl([0,T_\delta),H^m(I)\bigr).
$$
To that end, we derive \emph{a priori}  estimates that are uniform in $\varepsilon$ for the following energy function
        \[
        \begin{split}
            E(t)=& 1+\|\rho^{\varepsilon,\delta}\|_{L^2}^2+\|\partial_{x}^{m}\rho^{\varepsilon,\delta}\|_{L^2}^2+\left\| \frac{1}{\rho^{\varepsilon,\delta}}\right\|_{L^\infty} \\&+ \|A^{\varepsilon,\delta}\|_{L^2}^2+\|\partial_x^{m-1}A^{\varepsilon,\delta}\|_{L^2}^2+\left\| \frac{1}{A^{\varepsilon,\delta}}\right\|_{L^\infty}.
        \end{split}
        \]
        These bounds will guarantee both positivity and the required Sobolev--norm control up to time $T_\delta$.

        For the sake of clarity and
        focus we present here the proof for the case $m=3$ and  in each step we will comment how it can be adapted to the case $m\geq 4$.

       To simplify the notation and make the proof more readable, we introduce
        \begin{equation*}
            \begin{split}
                \Jrho=P\\
                \JA=\Lambda.
            \end{split}
        \end{equation*}
        We also observe that
        \begin{equation*}
            \begin{split}
                 \|\Lambda\|_{H^m}^2+\|P\|_{H^m}^2\leq E(t).
            \end{split}
        \end{equation*}

        In order to estimate  $\|\rho^{\varepsilon,\delta}\|_{L^2}^2$, we multiply the  equation of $\rho^{\varepsilon,\delta}$ by $\rho^{\varepsilon,\delta}$ and integrate by parts. Concretely
        \[
            \begin{split}
                \frac{1}{2}\frac{d}{dt}\int (\rho^{\varepsilon,\delta})^2\D x=&\beta\int  P^2\D x-\frac{\beta}{K}\int P^2 \Lambda \D x
                +\alpha(\mu-1)\int P^3\D x\\&-\mu\alpha\int P^2 \langle P \rangle\D x-\int P(\partial_x P)^2\D x.
            \end{split}
        \]
        Now, the positivity of  $P$ and $\Lambda$ along with the fact that $\mu<1$, ensure
        \[
            \begin{split}
                \frac{1}{2}\frac{d}{dt}\|\rhoed\|_{L^2}^2\leq&\beta\| P\|_{L^2}^2.
            \end{split}
        \]
        This implies
        \begin{equation*}
            \frac{d}{dt}\|\rho^{\varepsilon,\delta}\|_{L^2}^2\leq  \beta E(t).
        \end{equation*}

        Applying the identical strategy to the $L_2$ norm of third spatial derivative,
 yields a corresponding differential inequality and hence a uniform (in $\varepsilon$) bound on $\|\rho^{\varepsilon,\delta}(t)\|_{H^3}$:

            \begin{align*}
                \frac{1}{2}\frac{d}{dt}&\int_I(\partial_{x}^3\rho^{\varepsilon,\delta})^2\D x= \beta \int_I (\partial_x^3 P)^2dx -\frac{4\beta}{K} \int_I \partial_x^2 \Lambda P \partial_x P \partial_x^3 P \D x \\
                &-\frac{6\beta}{K}\int_I \Big( \partial_x \Lambda  (\partial_x P)^2 \partial_x^3 P   + \partial_x \Lambda   P \partial_x^2 P \partial_x^3 P \Big)\D x   \\
                &-\frac{\beta}{K} \int_I \Big( 6(\Lambda \partial_x P \partial_x^2 P \partial_x^3 P  +  \Lambda P (\partial_x^3 P)^2)    + \partial_x^2 \Lambda  P^2 \partial_x^4 P \Big)\D x  \\
                &+ 2\alpha(\mu-1)\int_I \Big(   P \partial_x^3 P  + 3 \partial_x P\partial_x^2 P \Big)\partial_x^3 P dx\\
                &- \mu \alpha \int_I \Big(    \partial_x^3 P\langle P \rangle  + 3\partial_x^2 P\langle  \partial_{x} P \rangle \Big)\partial_x^3 P \D x\\
                &- \mu \alpha \int_I  \Big(
                3\partial_x P\langle  \partial_{x}^2 P \rangle +
                P\langle  \partial_{x}^3 P \rangle\Big)\partial_x^3 P \D x\\
                &+8\int_I (\partial_x^3 P)^2\partial_x^2 P-\int_IP(\partial_x^4 P)^2 \D x.
            \end{align*}
        Observe that in the last identity integration by parts has been used twice. The first one applied to the term $\int_I \partial_x^3 P \partial_x^3 A P $ in order to avoid resorting to extra regularity on $A$ and the second one applied to the term coming from diffusion. More precisely, we have used
        \begin{equation*}
            \begin{split}
                \int_I \partial_x^3 P \partial_x^3 \Lambda P^2 \D x
                =-2\int_I \partial_x^4 P \partial_x^2 \Lambda P^2 \D x -\int_I \partial_x^3 P \partial_x^2 \Lambda P \partial_x P \D x,
            \end{split}
        \end{equation*}
        and
        \begin{equation*}
        \begin{split}
            \int_I \partial_x^3 P \partial_x^4(P\partial_x P) \D x= 8\int_I (\partial_x^3 P)^2 \partial_x^2 P \D x -\int_I P (\partial_x^4 P)^2\D x.
            \end{split}
        \end{equation*}

        By virtue of  the Sobolev embedding $H^m(I)\hookrightarrow L^\infty(I)$ for $m\ge1$, and a straightforward application of Hölder’s inequality, we obtain
        \begin{equation*}
            \begin{split}
                \frac{1}{2}\frac{d}{dt}\int_I &(\partial_{x}^3\rho^{\varepsilon,\delta})^2 \D x\leq \beta \|P\|_{H^3}^2+ 2(\alpha((1-\mu)+4\mu\|\Gamma\|_{L^1} )+4 )\|P\|_{H^3}^3\\
                &+\frac{12\beta}{K}\|\Lambda\|_{H^2}\|P\|_{H^3}^3-\frac{\beta}{K} \int_I \partial_x^2 \Lambda  P^2 \partial_x^4 P \D x-\int_I P( \partial_x^4 P)^2\D x.
            \end{split}
        \end{equation*}
        We now exploit the strict positivity of $\rho^{\varepsilon,\delta}(t,\cdot)$ on $[0,T_{\varepsilon,\delta})$ to absorb the troublesome fourth‐derivative term.  In particular, since
$$
\rho^{\varepsilon,\delta}(t,x)\ge c>0,
$$
we can dominate
$$
\int_I P\,\bigl(\partial_x^4 P\bigr)^2\,dx
$$
against itself with a negative sign, thereby controlling any  fourth‐order contributions. In particular, by Young inequality we deduce
        \begin{equation*}
            \begin{split}
                \frac{\beta}{K} \int_I \partial_x^2 \Lambda   P^2 \partial_x^4 P \D x -\int_I P(\partial_x^4 P)^2 \D x\leq \frac{1}{4}\left( \frac{\beta}{K}\right)^2 \int_I (\partial_{x}^2\Lambda)^{2} P^3 \D x.
            \end{split}
        \end{equation*}
        Therefore,
        \begin{equation*}
            \begin{split}
                \frac{1}{2}\frac{d}{dt}\int_I(\partial_{x}^3\rho^{\varepsilon,\delta})^2\D x \leq& \beta \|P\|_{H^3}^2+ 2(\alpha((1-\mu)+4\mu\|\Gamma\|_{L^1} )+4 )\|P\|_{H^3}^3\\
                &+ \|P\|_{H^3}^3\left( \frac{12\beta}{K}\|\Lambda\|_{H^2}
                +\frac{1}{4}\left( \frac{\beta}{K}\right)^2 \|\Lambda\|_{H^2}^2\right).
            \end{split}
        \end{equation*}
        By definition of $E(t)$, we deduce
        \begin{equation*}
            \frac{d}{dt}\|\partial_{x}^3\rho^{\varepsilon,\delta}\|_{L^2}^2\leq  CE^3(t) .
        \end{equation*}
       Hereafter, $C$ denotes a generic constant—whose value may change from line to line—but which always depends only on the fixed parameters $\alpha,\beta,\mu,\Gamma$ and $K$, 
 and is uniform in $\varepsilon$ and $\delta$.

        In the general case, it is also possible to deal with all terms except one  thanks to Hölder estimates, as we do previously. The one which requires  more attention is again  coming from the logistic interaction, that is,
        \begin{equation*}
            \int_I\partial_x^m\Lambda P^2\partial_x^mP \D x.
        \end{equation*}
        The idea again is to integrate by part to write
        \begin{equation*}
            \begin{split}
            \int _I\partial_x^m\Lambda P^2\partial_x^mP \D x =\int_I \partial_x^{m-1}\Lambda P\partial_x P\partial_x^mP \D x +\int_I \partial_x^{m-1}\Lambda P^2\partial_x^{m+1}P \D x.
            \end{split}
        \end{equation*}
        For the first term on the right‐hand side we simply invoke Hölder’s inequality.  The second term is then absorbed using Young’s inequality into the negative diffusion integral
$$
-\,\int_I P\,\bigl(\partial_x^{m+1}P\bigr)^2\,dx,
$$
which itself appears after integrating by parts the mixed derivative term
$$
\int_I \partial_x^{m}P\,\partial_x^{m+1}\bigl(P\,\partial_x P\bigr)\,dx.
$$

            Let us estimate $\|\frac{1}{\rho^{\varepsilon,\delta}}\|_{L^\infty}$. We use the definition of energy to directly bound the right-hand side of the equation for $\rho^{\varepsilon,\delta}$  and obtain
        \begin{equation*}
            \partial_t \frac{1}{\rhoed}=-\frac{\partial_t \rhoed}{(\rhoed)^2} \leq CE(t)^4 .
        \end{equation*}

                Given $h\in \mathbb{R}^+$, we integrate between $t$ and $t+h$ and divide by $h$ in both sides of the latter equation to get
        \begin{equation*}
            \frac{1}{h}\left( \frac{1}{\rho^{\varepsilon,\delta}}(t+h,x)-\frac{1}{\rho^{\varepsilon,\delta}}(t,x)\right) \leq \frac{C}{h}\int_{t}^{t+h}E(s)^4 ds.
        \end{equation*}
        Equivalently,
        \begin{equation*}
            \frac{1}{h} \frac{1}{\rho^{\varepsilon,\delta}}(t+h,x) \leq \frac{C}{h}\int_{t}^{t+h}E(s)^4 ds+\frac{1}{h}\frac{1}{\rho^{\varepsilon,\delta}}(t,x).
        \end{equation*}
        Taking the supremum in $x$ yields
$$
\frac{1}{h}\,\biggl\|\tfrac{1}{\rho^{\varepsilon,\delta}}(t+h,\cdot)\biggr\|_{L^\infty}
\;\le\;
\frac{C}{h}\int_{t}^{t+h}E(s)^{4}\,ds
\;+\;
\frac{1}{h}\,\biggl\|\tfrac{1}{\rho^{\varepsilon,\delta}}(t,\cdot)\biggr\|_{L^\infty}.
$$
Rearranging gives the incremental quotient estimate
$$
\frac{1}{h}\Bigl(
\bigl\|\tfrac{1}{\rho^{\varepsilon,\delta}}(t+h,\cdot)\bigr\|_{L^\infty}
\;-\;
\bigl\|\tfrac{1}{\rho^{\varepsilon,\delta}}(t,\cdot)\bigr\|_{L^\infty}
\Bigr)
\;\le\;
\frac{C}{h}\int_{t}^{t+h}E(s)^{4}\,ds.
$$

        Therefore, if $\frac{d}{dt} \left\|\frac{1}{\rhoed}\right\|_{L^\infty}$ exists, we can let $h$ to $0$ in this  inequality  and deduce
        \begin{equation*}
            \frac{d}{dt} \left\|\frac{1}{\rhoed}\right\|_{L^\infty} \leq CE^4(t).
        \end{equation*}

        The existence of $\frac{d}{dt} \left\|\frac{1}{\rhoed}\right\|_{L^\infty}$ is ensured almost everywhere by using Rademacher's theorem because the function $t\to \left\|\frac{1}{\rhoed}\right\|_{L^\infty}$ is Lipschitz.

        Indeed, given $t_1,t_2\in [0,T_{\varepsilon,\delta})$,
        \begin{equation*}
            \begin{split}
                \sup_{x}\frac{1}{\rhoed}(t_1,x) &=\sup_{x}\left\{\frac{1}{\rhoed}(t_1,x)-\frac{1}{\rhoed}(t_2,x)+\frac{1}{\rhoed}(t_2,x)\right\}\\
                &\leq \sup_{x}\left\{\frac{1}{\rhoed}(t_1,x)-\frac{1}{\rhoed}(t_2,x)\right\}+\sup_{x}\frac{1}{\rhoed}(t_2,x)\\
                &\leq
                \left|\sup_{x}\left\{\frac{1}{\rhoed}(t_1,x)-\frac{1}{\rhoed}(t_2,x)\right\}\right|+\sup_{x}\frac{1}{\rhoed}(t_2,x)\\
                &\leq
                \sup_{x}\left|\frac{1}{\rhoed}(t_1,x)-\frac{1}{\rhoed}(t_2,x)\right|+\sup_{x}\frac{1}{\rhoed}(t_2,x).
            \end{split}
        \end{equation*}
        If we repeat the same estimates after interchanging $t_1$ and $t_2$, we obtain:
        \begin{equation*}
            \left|\sup_{x}\frac{1}{\rhoed}(t_1,x)-\sup_{x}\frac{1}{\rhoed}(t_2,x)\right|\leq \sup_{x}\left|\frac{1}{\rhoed}(t_1,x)-\frac{1}{\rhoed}(t_2,x)\right|.
        \end{equation*}
        Since $\rho^{\varepsilon,\delta} \in C^1([0,T_{\varepsilon,\delta}),H^m(I))$ and remains strictly positive throughout this interval, we can invoke the mean value theorem on the right-hand side of the previous inequality. This yields the existence of some $s \in (t_1, t_2)$ such that
        \begin{equation*}
            \left|\sup_{x}\frac{1}{\rhoed}(t_1,x)-\sup_{x}\frac{1}{\rhoed}(t_2,x)\right|\leq \sup_{x}\left|\partial_t\frac{1}{\rhoed}(s,x)\right||t_1-t_2|.
        \end{equation*}
        The regularity of $\rho^{\varepsilon,\delta}$ and its sign again implies that $\partial_t\rho^{\varepsilon,\delta}(s,x)$ is uniformly bounded in $s$ and $x$. Hence,
        \begin{equation*}
            \left|\sup_{x}\frac{1}{\rhoed}(t_1,x)-\sup_{x}\frac{1}{\rhoed}(t_2,x)\right|\leq L|t_1-t_2|,
        \end{equation*}
        for some $L>0$.

The terms of $E(t)$ involving $A^{\varepsilon,\delta}$ are obtained using analogous methods to those employed for $\rho^{\varepsilon,\delta}$, including integration by parts, properties of mollifiers, and standard Sobolev embedding and interpolation inequalities. Putting everything together we have
    \begin{equation*}
        \frac{d}{dt}E(t)\leq CE^4(t),
    \end{equation*}
    which implies
    \begin{equation}\label{EnergyEstimate}
        E(t)\leq \frac{1}{\left( E(0)^{-3}-3Ct\right)^{1/3}}.
    \end{equation}

    In addition, the energy $E$ verifies the following inequality at $t=0$
    \begin{equation*}
        E(0)\leq \|\partial_x^3 \rho_0\|_{L^2}^2+\|\partial_x^2A_0\|_{L^2}^2 +\|\rho_0\|^2_{L^2}+\|A_0\|^2_{L^2}+\frac{2}{\delta},
    \end{equation*}
    and, since the right-hand side of this  inequality does not depend on $\varepsilon$, by a classical continuation of solutions argument, there exists $T_\delta>0$, independent of $\varepsilon$, such that $A^{\varepsilon,\delta}(t,\cdot)$ and $\rho^{\varepsilon,\delta}(t,\cdot)$ are positive in $[0,T_\delta)$ and
    \begin{equation*}\label{UniformlyBounded}
        \begin{split}
        &(A^{\varepsilon,\delta},\rho^{\varepsilon,\delta}) \text{ is uniformly bounded in } \varepsilon\\ &\text{ in the space } C([0,T_\delta),H^2(I))\times C([0,T_\delta),H^3(I))\,.
        \end{split}
    \end{equation*}
    Thanks to this property, after taking $L^2$-norm in the right-hand side of system \eqref{RegularizedEpsDelta} we also have
    \begin{equation*}
        \begin{split}
        &(\partial_t A^{\varepsilon,\delta},\partial_t\rho^{\varepsilon,\delta}) \text{ is uniformly bounded in } \varepsilon \\&\text{ in the space } C([0,T_\delta),L^2(I))\times C([0,T_\delta),L^2(I))\,.
        \end{split}
    \end{equation*}
    Given a  bounded interval $I$ and $s\in(0,1)$, the inclusions $H^2(I)\times H^3(I)\subset H^{2-s}(I)\times H^{3-s}(I)$ and $H^{2-s}(I)\times H^{3-s}(I)\subset\hspace{-0.02cm} L^2(I)\times L^2(I)$ are compact and continuous, respectively. Hence, a direct application of Aubin--Lions lemma along with a diagonal argument ensures that there exists a subsequence also denoted  $(A^{\varepsilon,\delta},\rho^{\varepsilon,\delta})$ such that
    \begin{equation}\label{Convergence}
        (A^{\varepsilon,\delta},\rho^{\varepsilon,\delta})\overset{\varepsilon\rightarrow 0}{\longrightarrow}(A^\delta,\rho^\delta)
        \end{equation}
         in  $C([0,T_\delta),H^{2-s}(I))\hspace{-0.03cm}\times C([0,T_\delta),H^{3-s}(I))$,
    for some pair of functions $(A^\delta,\rho^\delta)\in C([0,T_\delta),H^{2-s}(I))\hspace{-0.03cm}\times\hspace{-0.03cm} C([0,T_\delta),H^{3-s}(I))$. 	

Finally, the limiting functions inherit the uniform regularity in $\varepsilon$, and we deduce
    \begin{equation*}
	 (A^\delta,\rho^\delta)\in L^\infty([0,T_\delta),H^{2}(I))\times L^\infty([0,T_\delta),H^{3}(I)).
\end{equation*}
\end{proof}

We now verify that the limit pair $(A^\delta,\rho^\delta)$ obtained in Lemma~\ref{LemmaExistenceDelta1} indeed satisfies the regularized system \eqref{RegularizedDelta}. 

\begin{lemma}\label{LemmaExistenceDelta2}
	The sequence $(A^\delta,\rho^\delta)$ from Lemma \ref{LemmaExistenceDelta1} is a strictly positive solution to Problem \eqref{RegularizedDelta}.
\end{lemma}

\begin{proof}
     	In order to show that $(A^\delta,\rho^\delta)$ is a weak solution, we must check that it satisfies the following identities
\begin{equation}\label{IntDelta}
	\begin{split}
		-\int_{0}^T\int_{I}&A^\delta\partial_t \varphi \D x \D t = \int_{I}\Aed (0,x)\varphi (0,x)\D x \\ &+\int_{0}^T\int_{I} A^\delta\left(\alpha(1-\mu)\rho^\delta+\mu\alpha \langle \rho^\delta \rangle \right) \varphi \D x \D t  \\
		& +\int_{0}^T\int_{I} \left(\widetilde{\beta}A^\delta \left(1-\frac{\rho^\delta A^\delta}{\widetilde{K}}\right) \varphi -\rho^\delta\partial_{x}A^\delta\partial_{x}\varphi \right) \D x \D t,\\
		\end{split}
       \end{equation}
       and
       \begin{equation}\label{IntDelta1}
       \begin{split}
		&-\int_{0}^T\int_{I}\rho^\delta\partial_t\varphi \D x  \D t = \int_{I}\rhoed (0,x) \varphi (0,x) \D x
        \\&\hspace{0.5cm}+\int_{0}^T\int_{I}\beta \rho^\delta  \left(1-\frac{A^\delta\rho^\delta}{K} \right)\varphi \D x \D t \\
		&\hspace{0.5cm}-\int_{0}^T\int_{I}\left(\alpha(1-\mu)+\mu\alpha \langle \rho^\delta\rangle\varphi +\partial_{x}\rho^\delta\partial_{x}\varphi \right)\rho^\delta \D x \D t,\\
        \end{split}
        \end{equation}
for every $\varphi\in C^\infty_c([0,T)\times I)$.

The idea then is to  multiply each equation of \eqref{RegularizedEpsDelta} by $\varphi\in C^\infty_c([0,T)\times I)$ , integrate by parts and  show that the equalities \eqref{IntDelta}--\eqref{IntDelta1} are obtained when we send $\varepsilon$ to $0$.

More precisely, the rest of the proof is devoted to see that each term of the equalities
\begin{equation*}
	\begin{split}
		-\int_{0}^T&\int_{I}\Aed \partial_t \varphi \D x \D t = \int_{I}\Aed (0,x)\varphi (0,x)\D x\\
        &+ \int_{0}^T\int_{I} \mathcal{J_\varepsilon}\left( \JA\left(\alpha(1-\mu)\Jrho+\mu\alpha \langle \Jrho \rangle \right)\right)  \varphi \D x \D t   \\
		&\hspace{0.5cm}+\widetilde{\beta}\int_{0}^T\int_{I} \mathcal{J_\varepsilon}\left( \JA\left(1-\frac{\Jrho \JA}{\widetilde{K}}\right)\right) \varphi \D x \D t \\
		&\hspace{0.5cm}-\int_{0}^T\int_{I} \mathcal{J_\varepsilon}\left( \Jrho\JAuno\right)\partial_{x}\varphi \D x \D t,\\
		 \end{split}
\end{equation*}
and
\begin{equation*}
        \begin{split}
-\int_{0}^T&\int_{I}\rhoed \partial_t \varphi \D x \D t	= \int_{I}\rhoed (0,x) \varphi (0,x) \D x \\
&+\beta\int_{0}^T\int_{I}\mathcal{J_\varepsilon}\left( \Jrho  \left(1-\frac{\JA\Jrho}{K} \right)\right) \varphi \D x \D t \\	&\hspace{0.3cm}+\alpha\int_{0}^T\int_{I}\mathcal{J_\varepsilon}\left( \Jrho\Bigg( \hspace{-0.1cm}(\mu-1)\Jrho-\mu\langle \Jrho\rangle\Bigg)\hspace{-0.1cm}\right) \varphi \D x \D t\\
		&\hspace{0.3cm}-\int_{0}^T\int_{I}\mathcal{J_\varepsilon}\left( \Jrho\Jrhouno\right) \partial_{x}\varphi \D x \D t,\\
	\end{split}
\end{equation*}
converges, as $\varepsilon \to 0$, to the corresponding term of  \eqref{IntDelta} and \eqref{IntDelta1} for every $\varphi\in C^\infty_c([0,T)\times I)$

We next include  the proof of the convergence of the term 
\begin{equation*}
    \left|\int_{0}^T\hspace{-0.2cm}\int_{I}\mathcal{J_\varepsilon}\left(   \JA\left(\Jrho \right)^2\right) \varphi \D x \D t\hspace{-0.1cm}-\hspace{-0.1cm} \int_{0}^T\hspace{-0.2cm}\int_{I}  A^\delta\left(\rho^\delta \right)^2 \varphi \D x \D t\right|:=\mathcal{D}\overset{\varepsilon\to0}{\longrightarrow}0.
\end{equation*}
and leave the rest of them for the interested reader as they follow the same ideas.

To this end, by triangular inequality we  split the difference $\mathcal{D}$ into two terms  which will be estimated  separately. That is,
\begin{equation*}
	\begin{split}
		\mathcal{D}\leq&
		\int_0^T\int_{I}|\mathcal{J_\varepsilon}(\JA (\Jrho)^2)-\JA (\Jrho)^2||\varphi| \D x \D t\\ &+\int_{0}^T\int_{I}|\JA (\Jrho)^2 -A^\delta(\rho^\delta)^2||\varphi|\D x \D t\\
		&:=\mathcal{D}_{1}+\mathcal{D}_{2}. 
	\end{split}
\end{equation*}

In order to estimate $\mathcal{D}_1$, we use Hölder's inequality along with the fact that $
			\|\mathcal{J_\varepsilon}f-f\|_{H^{k-1}}   \leq \varepsilon C\|f\|_{H^k}$ for every $\varepsilon>0$, $k\in \mathbb{N}$ and $f\in H^k$. This results in
\begin{equation*}
    \mathcal{D}_1\leq \varepsilon T\|A^{\varepsilon,\delta} \|_{L^\infty H^1}\|\rho^{\varepsilon,\delta}\|_{L^\infty H^1}^2\|\varphi\|_{L^\infty L^2}.
\end{equation*}

To bound $\mathcal{D}_2$ ,we decompose it as follows 
\begin{equation*}
	\begin{split}
		\mathcal{D}_2\leq& \int_0^T\int_{I} |\JA|| (\Jrho)^2-(\rho^\delta)^2||\varphi| \D x \D t\\
       &+\int_0^T\int_{I} |\JA-A^\delta|| (\rho^\delta)^2||\varphi|\D x \D t.
       \end{split}
\end{equation*}
Then, triangular inequality and basic algebraic manipulations yield
\begin{equation*}
	\begin{split}
		\mathcal{D}_2\leq & \int_0^T\int_{I} |\JA|| \Jrho-\rho^{\varepsilon,\delta}|| \Jrho+\rho^\delta||\varphi|\D x \D t\\
		&+\int_0^T\int_{I} |\JA|| \rho^{\varepsilon,\delta}-\rho^\delta|| \Jrho+\rho^\delta||\varphi|\D x \D t\\
		&+\int_0^T\int_{I} |\JA-A^{\varepsilon,\delta}|| (\rho^\delta)^2||\varphi|dxdt\\
       &+\int_0^T\int_{I} |A^{\varepsilon,\delta}-A^\delta|| (\rho^\delta)^2||\varphi|\D x \D t.
	\end{split}
\end{equation*}
Hence, we employ again Hölder's inequality and  regularity properties of the operator $\mathcal{J}_\varepsilon$ to have
\begin{equation*}
	\begin{split}
		\mathcal{D}_2\leq & T\|A^{\varepsilon,\delta}\|_{L^\infty L^\infty}\Big(\|\rho^{\varepsilon,\delta}\|_{L^\infty H^1} \\
       &\hspace{1cm}+\|\rho^\delta\|_{L^\infty L^2}\Big)\|\varphi\|_{L^\infty L^\infty}(\varepsilon\|\rho^{\varepsilon,\delta}\|_{L^\infty H^1}+\|\rho^{\varepsilon,\delta}-\rho^\delta\|_{L^\infty L^2(\Omega)})\\
		&+T\Big(\varepsilon\|A^{\varepsilon,\delta}\|_{L^\infty H^1}+\|A^{\varepsilon,\delta}-A^\delta\|_{L^\infty L^2(\Omega)}\Big)\|\rho^\delta\|_{L^\infty H^1}^2\|\varphi\|_{L^\infty L^\infty}.
	\end{split}
\end{equation*}
Finally, we put together the bounds of $\mathcal{D}_1$ and  $\mathcal{D}_2$ and apply Lemma \ref{LemmaExistenceDelta1} to conclude
\begin{equation*}
    \mathcal{D}\overset{\varepsilon\to0}{\rightarrow}0,
\end{equation*}
which completes the proof.
\end{proof}

We now turn to the proof of Theorem \ref{TheoremExistence}. 
\begin{proof}[Proof of Theorem \ref{TheoremExistence}]

We define $T{_{\max}^\delta}$ as the maximal time for which $\rho^\delta$ belongs to the space $C([0, T{_{\max}^\delta)}, H^{m-s}(I))$. We claim that $\rho^\delta$ remains strictly positive throughout this interval.

	Indeed, assume by contradiction that
$$
\tau = \inf\{ t > 0 : \rho^{\delta}(t,x) = 0 \text{ for some } x \} < T_{\max}^\delta.
$$

Proceeding as in Lemma \ref{LemmaExistenceDelta1} for the map $t \mapsto \max_{x} \frac{1}{\rho^{\varepsilon,\delta}(t,x)}$, we can apply Rademacher's theorem to obtain that the map $t \mapsto m_{\rho^{\delta}(t)} := \min_x \rho^{\delta}(t,x)$ is differentiable almost everywhere. Moreover, since $\rho^{\delta}$ is continuous and the boundary conditions are periodic, it follows that for every $t$ there exists a point $x_t$ such that $m_{\rho^\delta}(t) = \rho^\delta(t, x_t)$.

We now compute the derivative of $m_{\rho^\delta}$. Since $x_t$ realizes the minimum, it follows that
	\begin{equation*}
		\begin{split}
			\frac{d m_{\rho^\delta}}{dt}(t)&=\lim_{h\to 0^+}\frac{\rho^\delta(t+h,x_{t+h})-\rho^\delta(t,x_{t})}{h}\\
			&=\lim_{h\to 0^+}\frac{\rho^\delta(t+h,x_{t+h})-\rho^\delta(t+h,x_{t})}{h}\\
            &\hspace{0.5cm}+\lim_{h\to 0^+}\frac{\rho^\delta(t+h,x_{t})-\rho^\delta(t,x_{t})}{h}\\
			&\leq \partial_t \rho^\delta (t,x_t),
		\end{split}
	\end{equation*} 
	and 
	\begin{equation*}
		\begin{split}
			\frac{d m_{\rho^\delta}}{dt}(t)&=\lim_{h\to 0^-}\frac{\rho^\delta(t+h,x_{t+h})-\rho^\delta(t,x_{t})}{h}\\
			&=\lim_{h\to 0^-}\frac{\rho^\delta(t+h,x_{t+h})-\rho^\delta(t+h,x_{t})}{h}\\
            &\hspace{0.5cm}+\lim_{h\to 0^-}\frac{\rho^\delta(t+h,x_{t})-\rho^\delta(t,x_{t})}{h}\\
			&\geq \partial_t \rho^\delta (t,x_t),
		\end{split}
	\end{equation*}
	Thus, one  finds
	\begin{equation}\label{Derivativem}
		\frac{d m_{\rho^\delta}}{dt}(t)=\frac{d}{dt} \rho^\delta (t,x_t).
	\end{equation}
	The regularity of solutions implies
	\begin{equation}\label{Positive}
		\begin{split}
			\frac{d m_{\rho^\delta} }{dt}(t)\geq& -m_{\rho^\delta}(t)\|\rho^\delta\|_{L^\infty}\Big( \alpha((1-\mu)+\mu\|\Gamma\|_{L^1})+\frac{\beta}{K}\|A^\delta\|_{L^\infty}  \Big)\\
			&+\partial_{x}\rho^\delta(t,x_t)\partial_{x}\rho^\delta(t,x_t)+\rho^\delta(t,x_t)\partial_{x}^2\rho^\delta(t,x_t),
		\end{split}
	\end{equation}
	for every $t\in[0,\tau]$.
	
	By definition of $x_t$ and $\tau$ it follows that
	\begin{equation*}
		\partial_{x}\rho^\delta(t,x_t)\partial_{x}\rho^\delta(t,x_t)=0\quad \text{ and } \quad\rho^\delta(t,x_t)\partial_{x}^2\rho^\delta(t,x_t)\geq 0.
	\end{equation*}

	Since $m_\rho(0)\geq \delta$, the former inequality contradicts the fact that $m_{\rho^\delta}(\tau)=0$.
	
	If $A$ enjoys sufficient regularity, say $m\ge4$, the same computations apply to $A$, proving its nonnegativity for as long as it exists in $C([0,T_{\max}^\delta),H^{m-1-s}(I))$.

	In what follows, we show that $T_{\max}^\delta$ is independent of $\delta$. This will ensure the required sign properties of $A$ and $\rho$ and permit the use of the Aubin–Lions lemma. To this end, we proceed as in the case of $E(t)$ to establish that the energy 
	\begin{equation*}
		\tilde{E}(t)=1+\|\partial_{x}^m\rho^{\delta}\|_{L^2}^2+\|\rho^{\delta}\|_{L^2}^2+ \|A^{\delta}\|_{L^2}^2+\|\partial_x^{m-1}A^{\delta}\|_{L^2}^2
	\end{equation*}
	verifies the differential inequality
	\begin{equation*}
		\frac{d \tilde{E}}{dt}(t)\leq C \tilde{E}^3(t).
	\end{equation*}
	This implies
	\begin{equation*}
		\tilde{E}(t)\leq \frac{1}{\left( \tilde{E}(0)^{-2}-2Ct\right)^{1/2}}.
	\end{equation*}
	Since we let $\delta \to 0$, the initial energy $\tilde{E}(0)$ is bounded by a constant $\tilde{E}_0$ independent of $\delta\in(0,1)$. Hence $\tilde{E}(t)$ remains uniformly bounded for all $\delta\in(0,1)$ on the interval $\bigl[0,\tfrac{1}{2C\tilde{E}_0}\bigr)$, and we conclude that
$$
T_{\max}^\delta \;\ge\; \frac{1}{2C\tilde{E}_0}
\quad\text{for every }\delta\in(0,1).
$$
Thanks to these uniform energy bounds, the Aubin--Lions lemma furnishes a subsequence (still indexed by $\delta$) along which $(A^\delta,\rho^\delta)$ converges in the appropriate spaces
	\begin{equation}\label{Convergence2}
		(A^\delta,\rho^\delta)\to(A,\rho) \text{ in } C([0,T),H^{m-1-s}(I))\hspace{-0.03cm}\times\hspace{-0.03cm} C([0,T),H^{m-s}(I))
	\end{equation}
	for every $s>0$ and for some $$(A,\rho)\in C([0,T),H^{m-1-s}(I))\hspace{-0.03cm}\times C([0,T),H^{m-s}(I)).$$
    
    Moreover, the strict positivity of $A^\delta$ and $\rho^\delta$ implies that their limits $A$ and $\rho$ are nonnegative.	

Now we proceed with the proof of the  continuity towards the initial data. In order to obtain such a result, we observe that the previously constructed approximate solution verifies
\[
(A^{\delta},\rho^\delta)\in C\bigl([0,T);H^{m-1-s}(I)\bigr)\times C\bigl([0,T);H^{m-s}(I)\bigr)
\]
due to the strict positivity of the solution and the uniform parabolicity of the diffusions in this case.  Then we obtain the estimate
\[
\begin{split}
\|A(t)-A_0\|_{H^{m-1}}^2&+\|\rho(t)-\rho_0\|_{H^{m}}^2\\
&\leq \|A(t)-A^\delta(t)\|_{H^{m-1}}^2+\|\rho(t)-\rho^\delta(t)\|_{H^{m}}^2\\& \quad +\|A^\delta_0-A^\delta(t)\|_{H^{m-1}}^2+\|\rho^\delta_0-\rho^\delta(t)\|_{H^{m}}^2\\& \quad+\|A^\delta_0-A_0\|_{H^{m-1}}^2+\|\rho^\delta_0-\rho_0\|_{H^{m}}^2.
\end{split}
\]
Then, taking $t$ sufficiently small and because of the properties of the approximate problem, we find that, for all $\varepsilon>0$, we can find $t,\delta$ small enough such that
\[
\|A(t)-A_0\|_{H^{m-1}}^2+\|\rho(t)-\rho_0\|_{H^{m}}^2\leq \varepsilon\\
\]
concluding the proof.
\end{proof}
\section{Uniqueness in the Local Setting}
	\label{sec:3}
	
    This section is devoted to showing the uniqueness of solutions for nonnegative initial data, thereby proving Theorem \ref{TheoremUniqueness}. The argument hinges on the $H^2$‐regularity of $\sqrt{\rho}$. Some of the ideas in this proof draw on techniques developed for porous‐medium equations; see also \cite{Fanelli2023}. As the lemma below demonstrates, it suffices to assume that the initial density’s square root, $\sqrt{\rho_0}$, belongs to $H^2$.

\begin{lemma}\label{RootLemma}
Under the hypotheses of Theorem \ref{TheoremExistence}, if 
\[
\sqrt{\rho_{0}}\in H^{2}(I),
\]
then there exists $T>0$ such that
\[
\sqrt{\rho}\,\in\, C \bigl([0,T);H^{2-s}(I)\bigr) \cap L^\infty \bigl([0,T);H^{2}(I)\bigr).
\]
\end{lemma}
\begin{proof}
	The proof again employs the Aubin--Lions lemma to extract compactness from a sequence of solutions to the regularized problem \eqref{RootEquation}. Since the regularization procedure and passage to the limit mirror those detailed above, we omit them here and concentrate solely on the a priori estimates that yield the uniform energy bounds
	\begin{equation*}
		E(t)=1+ \|\sqrt{\rho}\|_{L^2}^2+\|\partial_{x}^2\sqrt{\rho}\|_{L^2}^2.
	\end{equation*} 

    We introduce the change of variables
$
\eta = \sqrt{\rho}
$
 into system \eqref{Original}. It follows that the pair $(A,\eta)$ satisfies
	\begin{equation}\label{RootEquation}
	\begin{aligned}
		\partial_t \eta&=\eta^2 \hspace{-0,1cm}\left[\beta-\frac{\beta}{K}A\eta^2\hspace{-0,1cm}-\hspace{-0,1cm}\alpha \eta^2+\alpha\mu(\eta^2\hspace{-0,1cm}-\hspace{-0,1cm}\langle \eta^2\rangle)\right]
		\hspace{-0,1cm}+\hspace{-0,1cm}\eta(\partial_x \eta)^2 \hspace{-0,1cm}+ \hspace{-0,1cm}\partial_x(\eta^2\partial_x\eta)\\
		\partial_t A&=\alpha A\left((1-\mu)\eta^2+\mu\langle\eta^2\rangle \right)+A\widetilde{\beta}\left(1-\frac{A\eta^2}{\widetilde{K}}\right)
		+\partial_x(\eta^2\partial_xA).
	\end{aligned}
	\end{equation}
	
We now establish bounds on the $L^2$\nobreakdash-norm of $\eta$.
	\begin{equation*}
        \begin{split}
		\frac{1}{2}\frac{d}{dt} \int_I \eta^2 \D x=&\int_I \eta^2(\partial_x \eta)^2+\partial_x(\eta^2\partial_x\eta)\eta\\
		&+\eta^3\left[\beta-\frac{\beta}{K}A\eta^2-\alpha \eta^2+\alpha\mu(\eta^2-\langle\eta^2\rangle)\right]\D x
        \end{split}
	\end{equation*}
	Integrating by parts, taking into account that $\eta$ and $A$ are nonngeative functions, $\frac{\beta}{K}>0$ and $(1-\mu)>0$, we get
	\begin{align*}
		\frac{d}{dt} \int_I \frac{1}{2}\eta^2 \D x=&\int_I \eta^2(\partial_x \eta)^2-\eta^2(\partial_x\eta)^2
		\D x\\
        &+\int_I \eta^3\left[\beta-\frac{\beta}{K}A\eta^2-\alpha \eta^2+\alpha\mu(\eta^2-\langle \eta^2 \rangle)\right]\D x\\
		\leq&\int_I \eta^2(\partial_x \eta)^2+\beta\eta^3 \D x\\
		\leq& \|\eta\|_{L^\infty}^2\|\partial_x \eta\|_{L^2}^2+\beta\|\eta\|^2_{L^2}\|\eta\|_{L^\infty}\\
		\leq& C(1+\|\eta\|_{H^3})^2	\leq CE^2(t)
	\end{align*}
	
		The equation for $\partial_x^2 \eta$ is:
	\begin{align*}
		\partial_t \partial^2_x\eta=&17(\partial_x \eta)^2\partial_x^2\eta+8\eta(\partial_x^2\eta)^2+10\eta\partial_x\eta\partial_x^3\eta\\
        &+\eta^2\partial_x^4\eta
		+2\beta(\eta\partial^2_x\eta+(\partial_x\eta)^2)\\
		&-\frac{\beta}{K}(8\eta^3\partial_x\eta\partial_xA+\eta^4\partial^2_xA+4A\eta^3\partial^2_x\eta+12A\eta^2(\partial_x\eta)^2)\\
		&-4\alpha(1-\mu)(3\eta^2(\partial_x\eta)^2+\eta^3\partial^2_x\eta)\\
		&-2\alpha\mu(\eta^2\langle(\partial_x\eta)^2+\eta\partial^2_x\eta\rangle+((\partial_x\eta)^2+\eta\partial^2_x\eta)\langle \eta^2 \rangle)\\
		&-8\alpha\mu\eta\partial_x\eta\langle\eta\partial_x\eta\rangle
	\end{align*}
	
	We  proceed now to estimate $\|\partial^2_x\eta\|_{L^2}$. More precisely, we have
	\begin{align*}
		\frac{d}{dt} \int_I\frac{1}{2}(\partial^2_x\eta)^2 \D x=&\int_I 17(\partial_x \eta)^2(\partial_x^2\eta)^2+8\eta(\partial_x^2\eta)^3\D x\\
        &+\int_I 10\eta\partial_x\eta\partial^2_x\eta\partial_x^3\eta+\eta^2\partial_x^4\eta\partial^2_x\eta\D x\\
		&+2\beta\int_I(\eta(\partial^2_x\eta)^2+(\partial_x\eta)^2\partial^2_x\eta)\D x\\
		&-\frac{\beta}{K}\int_I (8\eta^3\partial_x\eta\partial^2_x\eta\partial_xA+\eta^4\partial^2_x\eta\partial^2_xA)\D x\\
		&-\frac{\beta}{K}\int_I (4A\eta^3(\partial^2_x\eta)^2+12A\eta^2(\partial_x\eta)^2\partial^2_x\eta)\D x\\
		&-4\alpha(1-\mu)\int_I (3\eta^2(\partial_x\eta)^2\partial^2_x\eta+\eta^3\partial^2_x\eta)\D x\\
		&-2\alpha\mu\int_I\eta^2\partial^2_x\eta\langle(\partial_x\eta)^2+\eta\partial^2_x\eta\rangle\D x\\
		&-2\alpha\mu\int_I((\partial_x\eta)^2\partial^2_x\eta+\eta(\partial^2_x\eta)^2)\langle \eta^2 \rangle\D x\\
		&-8\alpha\mu\int_I\eta\partial_x\eta\partial^2_x\eta\langle\eta\partial_x\eta\rangle\D x
	\end{align*}

    Applying Young’s and Hölder’s inequalities, together with the positivity of $A$ and $\eta$ and the Sobolev embedding, we obtain:
	\begin{align*}
		\frac{d}{dt}& \int_I\frac{1}{2}(\partial^2_x\eta)^2 \D x \leq C \Big( 1 + \|\eta\|^2_{H^2}\Big)^5 \\
		&+\int_I 17(\partial_x \eta)^2(\partial_x^2\eta)^2+8\eta(\partial_x^2\eta)^3+10\eta\partial_x\eta\partial^2_x\eta\partial_x^3\eta+\eta^2\partial_x^4\eta\partial^2_x\eta\D x
\end{align*}

	It remains to estimate the transport terms, which we handle by integration by parts followed by Young’s inequalities:
	\begin{align*}
		&\int_I 17(\partial_x \eta)^2(\partial_x^2\eta)^2+8\eta(\partial_x^2\eta)^3+10\eta\partial_x\eta\partial^2_x\eta\partial_x^3\eta+\eta^2\partial_x^4\eta\partial^2_x\eta\D x\\
		&=\int_I 9(\partial_x \eta)^2(\partial_x^2\eta)^2-8\eta\partial_x\eta\partial^2_x\eta\partial_x^3\eta-\eta^2(\partial_x^3\eta)^2\D x	\\
		&\leq\int_I 9(\partial_x\eta)^2(\partial_x^2\eta)^2+64(\partial_x\eta)^2(\partial^2_x\eta)^2+\eta^2(\partial_x^3\eta)^2-\eta^2(\partial_x^3\eta)^2\D x\\
		&=\int_I 73(\partial_x\eta)^2(\partial_x^2\eta)^2\D x\leq73\|\partial_x\eta\|^2_{L^\infty}\|\partial_x^2\eta\|^2_{L^2}
	\end{align*}
	By combining the above estimates and applying Sobolev embedding theorems, we obtain:
	$$\partial_t \int_I\frac{1}{2}(\partial^2_x\eta)^2 \D x\leq C(1+\|\eta\|^2_{H^2})^5=CE^5(t)$$
	Thus, we arrive at the desired energy estimate:
	$$\frac{d}{dt}E(t)\leq C E^5(t).$$
\end{proof}
We are now in a position to conclude the proof of Theorem~\ref{TheoremUniqueness}.
	\begin{proof}[Proof of Theorem \ref{TheoremUniqueness}]
	Let $(A_1,\rho_1)$ and $(A_2,\rho_2)$ be two solutions with the same initial data. Define
\[
\Delta = A_1 - A_2,\quad R = \rho_1 - \rho_2,\quad N = \eta_1 - \eta_2,
\]
where $\eta_i = \sqrt{\rho_i}$ for $i=1,2$.
	The equation for the evolution of $\Delta$ can be written as 
	\begin{equation*}
		\begin{split}
			\partial_t \Delta=&\alpha(1-\mu)(\Delta\rho_1+A_2R)+\alpha\mu(\Delta\langle \rho_1 \rangle+ A_2\langle R \rangle)+ \widetilde{\beta}\Delta\\
            &-\frac{\widetilde{\beta}}{\widetilde{K}}((A_1^2 -A_2^2)\rho_1+RA_2^2)+\partial_x(\rho_1\partial_x \Delta)+\partial_x(R\partial_x A_2).
		\end{split}
	\end{equation*}
	In the  same way, we  have
	\begin{equation*}
		\begin{split}
			\partial_t \Delta=&\alpha(1-\mu)(\Delta\rho_2+A_1R)+\alpha\mu(\Delta\langle \rho_2 \rangle+ A_1\langle R \rangle)+ \widetilde{\beta}\Delta\\&-\frac{\widetilde{\beta}}{\widetilde{K}}((A_1^2 -A_2^2)\rho_2+RA_1^2)
			+\partial_x(\rho_2\partial_x \Delta)+\partial_x(R\partial_x A_1).
		\end{split}
	\end{equation*}
	Consequently, $\Delta$ satisfies the symmetrized equation 
	\begin{equation}\label{EquationDelta}
		\begin{split}
			2\partial_t \Delta=&\alpha(1-\mu)(\Delta(\rho_1+\rho_2)+(A_1+A_2)R)\\
            &+\alpha\mu(\Delta\langle \rho_1+\rho_2 \rangle+ (A_1+A_2)\langle R \rangle)\\
			&+ 2\widetilde{\beta}\Delta-\frac{\widetilde{\beta}}{\widetilde{K}}((A_1^2 -A_2^2)(\rho_1+\rho_2)+R(A_1^2+A_2^2))\\
			&+\partial_x((\rho_1+\rho_2)\partial_x \Delta)+\partial_x(R\partial_x (A_1+A_2)).
		\end{split}
	\end{equation}
	
	Multiplying the symmetrized equation by $\Delta$ and integrating over $I$, we deduce that $\|\Delta\|_{L^2}^2$ satisfies the differential inequality 
	\begin{equation*}
		\frac{d}{dt} \|\Delta\|_{L^2}^2\leq \int_I \Delta(\mathcal{A}+\mathcal{B}) \D x,
	\end{equation*}
	where $\mathcal{A}$ includes all the terms of \eqref{EquationDelta} which do not involve spatial derivatives, that is
	\begin{equation*}
		\begin{split}
			\mathcal{A}=&\alpha(1-\mu)(\Delta(\rho_1+\rho_2)+(A_1+A_2)R)\\
            &+\alpha\mu(\Delta\langle \rho_1+\rho_2 \rangle+ (A_1+A_2)\langle R \rangle)\\
			&+2\widetilde{\beta}\Delta-\frac{\widetilde{\beta}}{\widetilde{K}}((A_1^2 -A_2^2)(\rho_1+\rho_2)+R(A_1^2+A_2^2)),
		\end{split}
	\end{equation*}  
	and $\mathcal{B}$ stands for the rest of terms:
	\begin{equation*}
		\mathcal{B}=\partial_x((\rho_1+\rho_2)\partial_x \Delta)+\partial_x(R\partial_x (A_1+A_2)).
	\end{equation*}
	
	The integral
\[
\int_I \Delta\,\mathcal{A}\,\mathrm{d}x
\]
can be bounded using Young’s and Hölder’s inequalities as follows:
	\begin{equation*}
		\begin{split} 
			\int_I \Delta\mathcal{A}\D x\leq& \alpha(1-\mu)\left(\|\Delta\|_{L^2}^2\|\rho_1+\rho_2\|_{L^\infty}+\frac{\|\Delta\|_{L^2}^2}{2}\right)\\
			&+\frac{\alpha(1-\mu)}{2}(\|A_1+A_2\|_{L^\infty}^2\|\eta_1+\eta_2\|_{L^\infty}^2\|N\|_{L^2}^2)\\
			&+\alpha\mu\left( \|\Delta\|_{L^2}^2\|\Gamma\|_{L^1}\|\rho_1+\rho_2\|_{L^\infty}+\frac{\|\Delta\|_{L^2}^2}{2}\right)\\
			&+\frac{\alpha\mu}{2}(\|A_1+A_2\|_{L^\infty}^2\|\Gamma\|_{L^1}^2\|N\|_{L^2}^2\|\eta_1+\eta_2\|_{L^\infty}^2)\\
			&+2\widetilde{\beta}\|\Delta\|_{L^2}^2+\frac{\widetilde{\beta}}{\widetilde{K}}\|\Delta\|_{L^2}^2\|A_1+A_2\|_{L^\infty}\|\rho_1+\rho_2\|_{L^\infty}\\
			&+\frac{\widetilde{\beta}}{2\widetilde{K}}\left( \|\Delta\|_{L^2}^2+ \|N\|_{L^2}^2\|A_1^2+A_2^2\|_{L^\infty}^2\|\eta_1+\eta_2\|_{L^\infty}^2  \right).
		\end{split}
	\end{equation*}
	We now integrate by parts in $\int_I \Delta\mathcal{B} \D x$ to derive
	\begin{equation*}
		\begin{split}
			\int_I \Delta\mathcal{B}\D x= -\int_I (\rho_1+\rho_2)(\partial_x \Delta)^2\D x-\int_I R\partial_x(A_1+A_2)\partial_x\Delta \D x.
		\end{split}
	\end{equation*}
	A direct application of Young inequality along with the identity $R=(\sqrt{\rho_1}-\sqrt{\rho_2})(\sqrt{\rho_1}+\sqrt{\rho_2})$ yields
	\begin{equation*}
		\begin{split}
			\int_I \Delta\mathcal{B} \D x\leq& -\frac{1}{2}\int_I(\rho_1+\rho_2)(\partial_x \Delta)^2 \D x\\
            &+ \|N\|_{L^2}^2\|\partial_x(A_1+A_2)\|_{L^\infty}^2\|\partial_x\Delta\|_{L^\infty}^2.
		\end{split}
	\end{equation*}
	Regarding the estimates  of $\|N\|_{L^2}^2$, we can proceed in a similar way as we did for $\|\Delta\|_{L^2}^2$ to obtain
	\begin{equation*}
		\begin{split}
			2\frac{d}{dt}N=&\alpha(\mu-1)N(\eta_1(\eta_1+\eta_2)+\eta_2^2)\\
            &+\frac{\alpha}{2}\mu(N\langle \eta_1^2+\eta_2^2 \rangle+\langle N(\eta_1+\eta_2) \rangle(\eta_1+\eta_2))+\beta N(\eta_1+\eta_2)\\
            &-\frac{\beta}{2K}(N(\eta_1(\eta_1+\eta_2)+\eta_2^2)(A_1+A_2)+\Delta(\eta_1^3+\eta_2^3))\\
			&+N((\partial_x \eta_1)^2+(\partial_x \eta_2)^2)+\partial_xN(\partial_x \eta_1+\partial_x \eta_2)(\eta_1+\eta_2)\\
			&+ \partial_x(N(\eta_1+\eta_2)\partial_x(\eta_1+\eta_2))+\partial_x(\partial_xN(\eta_1^2+\eta_2^2)).
		\end{split}
	\end{equation*}
	Hence, we have
	\begin{equation*}
		\begin{split}
			\frac{d}{dt}\|N\|_{L^2}^2\leq &\alpha(\mu-1)\|N\|_{L^2}^2\|\eta_1(\eta_1+\eta_2)+\eta_2^2\|_{L^\infty}\\
            &+\frac{\alpha}{2}\mu\|N\|_{L^2}^2\|\Gamma\|_{L^1}\|\eta_1^2+\eta_2^2\|_{L^\infty}\\
			&+\frac{\alpha}{4}\mu \|N\|_{L^2}^2(1+ \|\Gamma\|_{L^1}^2\|\eta_1+\eta_2\|_{L^\infty}^4)+\beta\|N\|_{L^2}^2\|\eta_1+\eta_2\|_{L^\infty}  \\
			&+\frac{\beta}{2K}\|N\|_{L^2}^2\|\eta_1(\eta_1+\eta_2)+\eta_2^2\|_{L^\infty}\|(A_1+A_2)\|_{L^\infty}\\
			&+\frac{\beta}{4K}\|\eta_1^3+\eta_2^3\|_{	L^\infty}(\|\Delta\|_{L^2}^2+\|N\|_{L^2}^2)\\
            &+\|N\|_{L^2}^2\|(\partial_x \eta_1)^2+(\partial_x \eta_2)^2\|_{L^\infty}\\
			&+\int_I N\partial_xN\partial_x(\eta_1+ \eta_2)(\eta_1+\eta_2)\D x\\
            &+\int N \partial_x(N(\eta_1+\eta_2)\partial_x(\eta_1+\eta_2))\D x\\
			&+\int_I \D x N\partial_x(\partial_xN(\eta_1^2+\eta_2^2)) \D x.
		\end{split}
	\end{equation*}
	To control the spatial-derivative contributions in the previous estimate, we again exploit the decay afforded by the final term. In particular, integrating by parts and applying Young’s inequality yields
    \begin{equation*}
		\begin{split}
			\int_I N\partial_xN\partial_x(\eta_1+ \eta_2)(\eta_1+\eta_2) \D x+&\int N \partial_x(N(\eta_1+\eta_2)\partial_x(\eta_1+\eta_2)) \D x
			\\
            &+\int_I N\partial_x(\partial_xN(\eta_1^2+\eta_2^2)) \D x\\
			&\hspace{-4.5cm}\leq
			\left(\frac{1}{c}-1\right)\int_I (\partial_xN)^2(\eta_1^2+\eta_2^2) \D x+2c\|N\|_{L^2}^2\|\partial_x(\eta_1+\eta_2)\|_{L^\infty}^2 ,
		\end{split}
	\end{equation*}
	for every $c>0$.
	
	Collecting the above estimates, along with the regularity proved for $A_i,\rho_i,\eta_i$ with $i=1,2$,  show that there exists a constant $C$, depending only on the system parameters and the initial norms $\|A_i(0)\|_{H^2}$, $\|\eta_i(0)\|_{H^2}$ for $i=1,2$, such that
\[
\frac{d}{dt}\bigl(\|\Delta\|_{L^2}^2 + \|N\|_{L^2}^2\bigr)
\;\le\;
C\bigl(\|\Delta\|_{L^2}^2 + \|N\|_{L^2}^2\bigr).
\]
An application of Grönwall’s inequality then implies $\Delta\equiv0$, $N\equiv0$, and hence establishes uniqueness, completing  the proof.
\end{proof}

\section{Non-Global Existence in Time}
\label{sec4}

In this section, we show that the solution constructed in Theorems \ref{TheoremExistence} and \ref{TheoremUniqueness} may develop a finite‐time singularity under appropriate initial conditions. Our proof begins by choosing a symmetric initial density profile $\rho_{0}$ that vanishes at some point and whose second spatial derivative there is sufficiently large. By monitoring the evolution of $\partial_{x}^{2}\rho$ at that point, we derive a differential inequality which shows that, if the initial curvature exceeds a critical threshold, the solution cannot remain regular for all $t>0$.

To establish the necessary differential inequality, we first require an explicit bound on 
\[
\|\rho\|_{L^{\infty}([0,T)\times I)}.
\]
The next lemma provides this estimate.

\begin{lemma}\label{LinftyLinfty}
	Let 
\[
(A_0,\rho_0)\in H^3(I)\times H^4(I),
\qquad 
\mu\in[0,1),
\]
and let $T>0$ be the maximal existence time for the solution 
$\rho\in C\bigl([0,T);H^{4-s}(I)\bigr)$. Then
\[
\|\rho\|_{L^\infty\bigl([0,T)\times I\bigr)}
\;\le\;
\frac{\beta}{\alpha\,(1-\mu)}.
\]
\end{lemma}
\begin{proof}
	Repeating the Rademacher‐theorem‐based argument which we used to bound 
$\|\rho^{\varepsilon,\delta}(t,\cdot)\|_{L^\infty}$ 
and to derive identity \eqref{Derivativem} then yields
\begin{equation*}
	\frac{d} {dt}\|\rho(t,\cdot)\|_{L^\infty}=\partial_t\rho(t,x_t),
\end{equation*}
where $x_t$ is the point of maximum.

Invoking the positivity of $\rho$ and $A$ in the $\rho$-equation \eqref{Original}, we obtain
\begin{equation*}
	\partial_t\rho(t,x_t)\leq \beta\rho(t,x_t)+\alpha(\mu-1)\rho^2(t,x_t).
\end{equation*}
Hence, since $\mu<1$, we conclude
\[
\|\rho\|_{L^\infty\bigl([0,T)\times I\bigr)}
\;\le\;
\frac{\beta}{\alpha\,(1-\mu)}.
\]
\end{proof}
We are now prepared to state and prove the principal result of this section.

	\begin{theorem}[Finite‐time blow‐up]\label{TheoremBlowUp}
Let $\mu<1$, $m\ge4$, and let 
\[
(A_0,\rho_0)\;\in\;H^{m-1}(I)\times H^m(I)
\]
be a pair of even functions satisfying
\begin{enumerate}
  \item $\rho_0(0)=0$,
  \item $\sqrt{\rho_0}\in H^2(I)$,
  \item $\displaystyle \partial_{x}^2\rho_0(0) \;>\;\frac{\mu\,\beta}{1-\mu}\,\|\Gamma\|_{L^1(I)}$.
\end{enumerate}
Let $T$ be the maximal time of existence of the corresponding solution \[(A,\rho)\in C\bigl([0,T);H^{3-s}(I)\bigr)\times C\bigl([0,T);H^{4-s}(I)\bigr). \] Then
\[
T<\infty,
\]
and, provided no other singularity develops earlier, there exists $\tau<T$ such that
\[
\lim_{t\to\tau^-}\partial_x^2\rho(t,0)=+\infty.
\]
\end{theorem}

\begin{figure}[ht] 
		\includegraphics[width=0.48\textwidth]{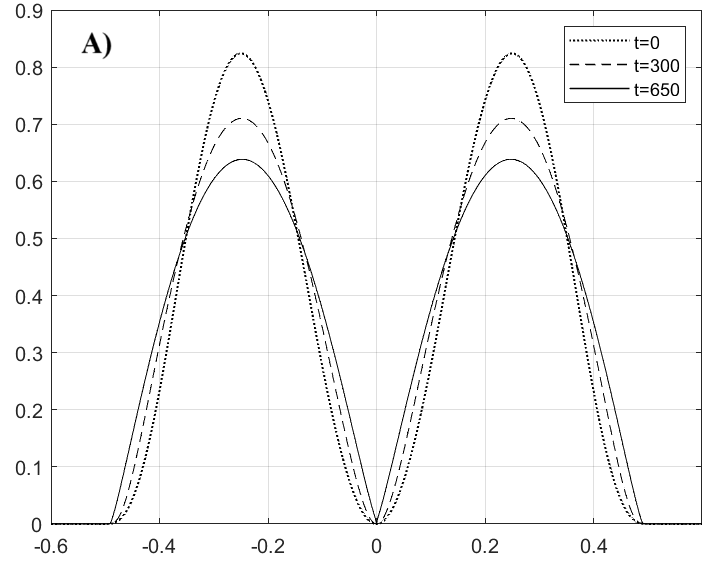}\
		\includegraphics[width=0.48\textwidth]{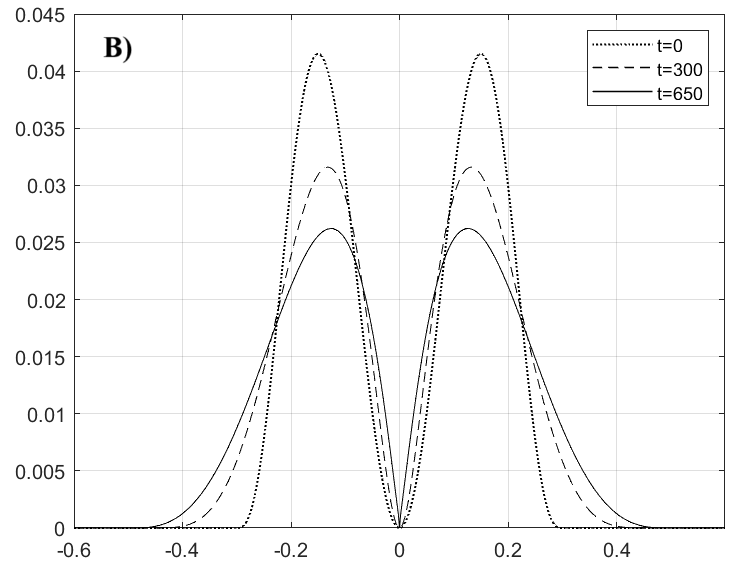}

		\includegraphics[width=0.48\textwidth]{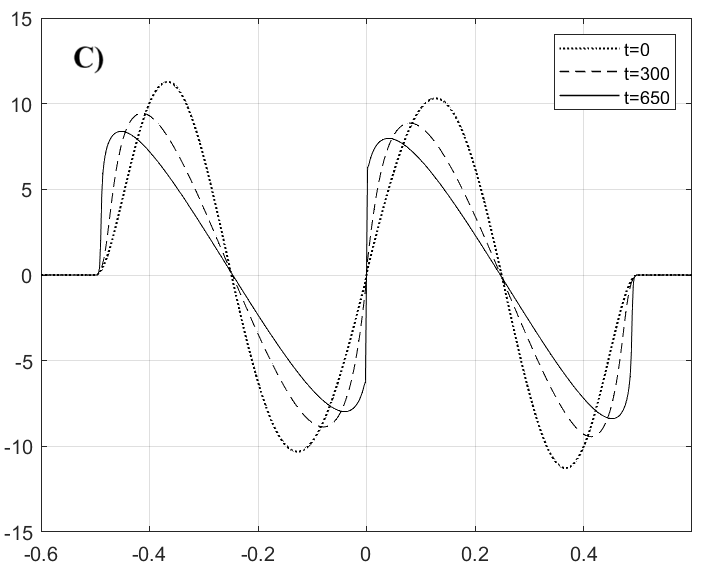}
	\caption{Representation of the density $\rho$ A), the area $A$  B), and $\partial_x \rho$  C). The time scale is on the order of $2.5 \times 10^{-6}$. The parameters of the system and the initial data: $\alpha = 1$, $\mu = 0.5$, $\beta=0.75$, $\widetilde{\beta}=0.5$, $K=1$, $\widetilde{K}=0.5$, $\Gamma=\chi_{[-\epsilon,\epsilon]}$, $\epsilon=0.05$, $\rho_0=-2000(x+0.5)^3x^2(x-0.5)^3\chi_{[-0.5,0.5]}$, $A_0=-6000(x+0.3)^3x^2(x-0.3)^3\chi_{[-0.3,0.3]}$.Study of finite-time blow-up for even initial data, as established in Theorem \ref{TheoremBlowUp}. We observe that, in finite time, the profile of $\rho$ at $x = 0$ evolves from a smooth shape to a singular one, resembling the Barenblatt-type profile. This transition leads to the formation of a vertical slope in the gradient at $x = 0$, which implies that the second derivative at that point becomes infinite.
    } \label{fig-1}
\end{figure}

\begin{proof}
	Since the initial data $(A_0,\rho_0)$ are even and the solution is unique (Theorem \ref{TheoremUniqueness}),    then $(A,\rho)(t,x)$ remains even in $x$ for all $t\in[0,T)$. Hence,  
\[
\partial_x\rho(t,0)=0
\quad\text{for all }t\in[0,T).
\]
	
	On the other hand, from the $\rho$-equation \eqref{Original} we deduce
	\begin{equation*}
		\rho(t,x)=\rho_0(x)e^{\beta\ \left(1-\frac{A\rho}{K} \right)-\alpha\rho+\mu(\rho-\langle \rho\rangle)+ \partial_{x}^2\rho+4(\partial_{x}\sqrt{\rho})^2}.
	\end{equation*}

	It follows that 
\[
\rho(t,0)=0
\quad\text{for all }t\in[0,T).
\]
Since $\rho(t,x)\ge0$, these conditions imply that $x=0$ is a global minimum of $\rho(t,\cdot)$. Therefore,
\[
\partial_{x}^{2}\rho(t,0)\;\ge\;0
\quad\text{for all }t\in[0,T).
\]

	On the other hand, $\partial_{x}^2 \rho$ satisfies 
	\begin{equation*}
		\begin{split}
			\partial_t \partial_{x}^2\rho=&\beta\partial_{x}^2\rho\left(1-\frac{A\rho}{k} \right)+2\beta \partial_{x}\rho\partial_{x}\left(1-\frac{A\rho}{K} \right)+\beta \rho\partial_{x}^2\left(1-\frac{A\rho}{K} \right)\\
			&+\mu \alpha \partial_{x}^2\rho\left(\rho-\langle \rho\rangle \right) +2\mu \alpha \partial_{x}\rho\left(\rho-\langle \rho\rangle \right)_{x}+\mu \alpha \rho\partial_{x}^2\left(\rho-\langle \rho\rangle \right)\\
			&-2\alpha\partial_{x}^2\rho\rho-2\alpha(\partial_{x}\rho)^2+ \partial_x^3\rho \partial_x\rho +3 (\partial_{x}^2 \rho)^2+3 \partial_x \rho \partial_{x}^3 \rho + \rho \partial_{x}^4 \rho  .
		\end{split}
	\end{equation*}
	Evaluating this expression at $x=0$ and using $\rho(t,0)=\partial_x\rho(t,0)=0$ together with $\partial_x^2\rho(t,0)\ge0$ yields
	\begin{equation}\label{SecondDerivative0}
		\partial_t\partial_{x}^2\rho(t,0)\geq(\partial_{x}^2\rho(t,0))^2-\alpha\mu\partial_{x}^2\rho(t,0)\langle\rho\rangle(t,0),
	\end{equation}
	for every $t\in[0,T)$.
	
	We now apply Lemma \ref{LinftyLinfty} in order to prove that the term $\langle \rho(t,0) \rangle$ verifies the bound
	\begin{equation*}
		|\langle \rho \rangle|(t,0)\leq \|\Gamma\|_{L^1}\|\rho\|_{L^\infty}\leq \|\Gamma\|_{L^1}\frac{\beta}{\alpha(1-\mu)}.
	\end{equation*}
	
	By introducing this estimate into inequality \eqref{SecondDerivative0}, we deduce 
	\begin{equation}
		\partial_t\partial_{x}^2\rho(t,0)\geq(\partial_{x}^2\rho(t,0))^2-\|\Gamma\|_{L^1}\beta\frac{\mu}{1-\mu}\partial_{x}^2\rho(t,0),
	\end{equation}
	for every $t\in[0,T)$.
	
	Finally, since
\[
\partial_{x}^2\rho_0(0) \;>\;\frac{\mu\,\beta}{1-\mu}\,\|\Gamma\|_{L^1},
\]
the differential inequality derived above entails that $\partial_x^2\rho(t,0)$ becomes unbounded in finite time (see Fig. \ref{fig-1}). This completes the proof.  
\end{proof}

\section{Dynamics of the Spatial Support}
\label{sec5}

We now examine the evolution of the spatial support of solutions. Specifically, we ask whether the positivity set 
$\mathrm{supp}\,\rho(t,\cdot)$ can grow beyond its initial domain or remains invariant. Under the regularity assumption
\[
\sqrt{\rho_0}\,\in\,H^2(I),
\]
one shows that
\[
\mathrm{supp}\,\rho(t,\cdot)
=\mathrm{supp}\,\rho_0,
\qquad
\forall\,t\in[0,T).
\]
Therefore, despite the nonlocal and nonlinear interactions in the system, the density cannot instantaneously occupy new spatial regions.

On the one hand, we show that if 
\[
\operatorname{supp}A_{0}\;\subset\;\operatorname{supp}\rho_{0},
\]
then for every $t\in[0,T)$ one has
\[
\operatorname{supp}A(t,\cdot)
=\operatorname{supp}\rho(t,\cdot).
\]
Biologically, this expresses the fact that the density cannot be positive in regions where there is no area for it to occupy. The following results make this statement precise.

\begin{proposition}\label{SupportRho}
Suppose
\[
(A_0,\rho_0)\,\in\,H^3(I)\times H^4(I),
\qquad
\sqrt{\rho_0}\,\in\,H^2(I),
\]
and let $T>0$ be such that the statements of Theorems \ref{TheoremExistence} and \ref{TheoremUniqueness} are satisfied.
Then for every $t\in[0,T)$,
\[
\operatorname{supp}\rho(t,\cdot)
\;=\;
\operatorname{supp}\rho_0.
\]
\end{proposition}

\begin{proof}
	Writing
\[
\rho(t,x)
= \rho_0(x)\exp\Bigl[\,
\beta\Bigl(1-\tfrac{A\rho}{K}\Bigr)
-\alpha\,\rho
+\mu\bigl(\rho-\langle\rho\rangle\bigr)
+\partial_x^2\rho
+4\bigl(\partial_x\sqrt{\rho}\bigr)^2
\Bigr],
\]
and recalling the regularity of $A,\rho,\eta$ from Theorem \ref{TheoremExistence} and Lemma \ref{RootLemma}, we see that
$\rho(t,x)=0$ if and only if $\rho_0(x)=0$.
\end{proof}

\begin{proposition}\label{SupportA}
Under the hypotheses of Theorem \ref{TheoremExistence} with $m\ge4$, the following properties hold on $[0,T)$:
\begin{enumerate}[label=\roman*)]
\item Lower bound on $A$: Define
   \[
     m_A(t) \;=\; \inf_{x\in I} A(t,x).
   \]
   Then there exists a constant
   \[
     C \;=\; C\bigl(\|A\|_{L^\infty([0,T)\times I)},\,\|\rho\|_{L^\infty([0,T)\times I)}\bigr)
   \]
   such that
   \[
     m_A(t)\;\ge\;m_A(0)\,e^{-Ct},
     \quad
     \forall\,t\in[0,T).
   \]

\item Invariance of support: Assume that $\operatorname{supp}\rho_0$ is a closed interval whose interior satisfies $\rho_0(x)>0$, for all $x\in\operatorname{int}(\operatorname{supp}\rho_0)$, and that 
$\operatorname{supp}A_0\subset\operatorname{supp}\rho_0$.  Then for every $t\in[0,T)$, we have 
   \[
     \operatorname{supp}A(t,\cdot)
     \;=\;
     \operatorname{supp}\rho_0.
   \]
   \end{enumerate}
\end{proposition}

\begin{figure}[ht] 
	\includegraphics[width=0.48\textwidth]{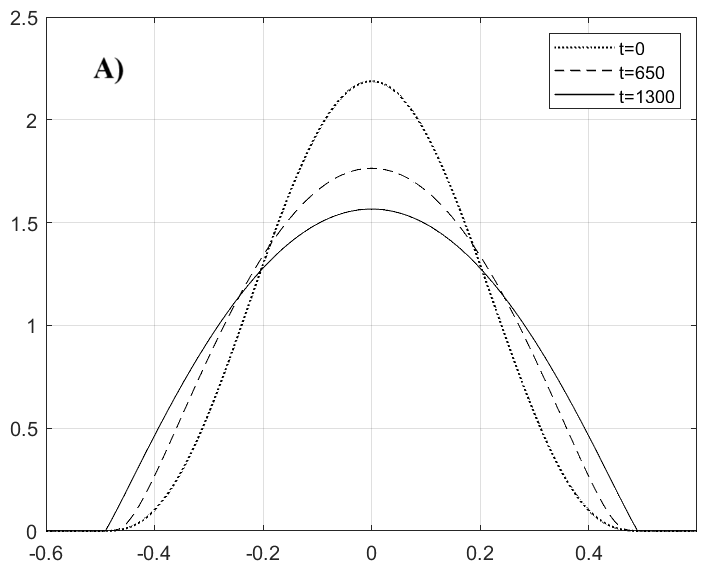}\
	\includegraphics[width=0.48\textwidth]{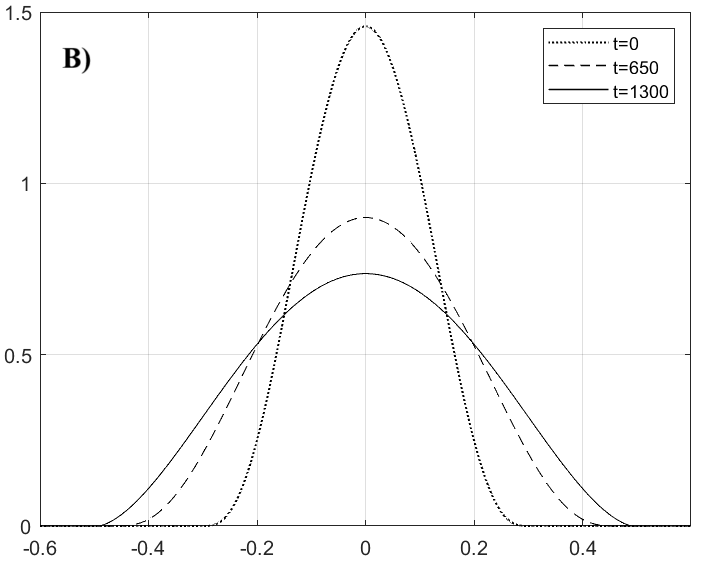}
	
	\includegraphics[width=0.48\textwidth]{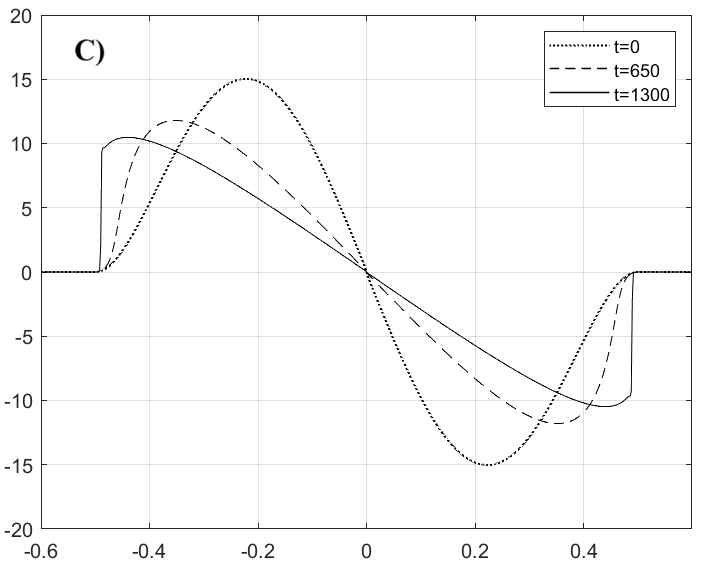}
	\caption{Representation of the density $\rho$  A), the area $A$  B), and $\partial_x \rho$  C). The time scale is on the order of $2.5 \times 10^{-6}$. The parameters of the system and the initial data: $\alpha = 1$, $\mu = 0.5$, $\beta=0.75$, $\widetilde{\beta}=0.5$, $K=1$, $\widetilde{K}=0.5$, $\Gamma=\chi_{[-\epsilon,\epsilon]}$, $\epsilon=0.05$, $\rho_0=-140(x+0.5)^3(x-0.5)^3\chi_{[-0.5,0.5]}$, $A_0=-2000(x+0.3)^3(x-0.3)^3\chi_{[-0.3,0.3]}$. Study of  the dynamics of  supports, as analyzed in  Theorem \ref{TheoremSupport}. We observe that the support of $\rho$ remains unchanged over time until the blow-up of the derivative occurs. During this waiting period, the support of $A$ evolves from coinciding with the initial support to expands instantaneously and matches that of $\rho$, after which it also remains unchanged. As $\rho$ begins to change its shape towards a Barenblatt-type profile, its derivative becomes increasingly steep. Eventually, at the boundaries of the support of $\rho$, the derivative  starts exhibiting a near-vertical slope, indicating that the second derivative blows up at the boundary.} \label{fig-2}
\end{figure}

\begin{proof}
  To prove $i)$, we return to the estimate underlying \eqref{Positive}. Let $x_t\in I$ be such that $A(t,x_t)=m_A(t)$. Then, from the evolution equation for $A$,
\[
\frac{d}{dt}m_A(t)
=\partial_t A(t,x_t)
\;\ge\;
-\,\frac{\tilde \beta}{
\tilde k} \|A\|_{L^\infty([0,T)\times I)}\;\|\rho\|_{L^\infty([0,T)\times I)}\;m_A(t).
\]
Setting 
\[
C \;=\;\frac{\tilde \beta}{
\tilde k} \|A\|_{L^\infty([0,T)\times I)}\,\|\rho\|_{L^\infty([0,T)\times I)},
\]
this gives
\[
\frac{d}{dt}m_A(t)\;\ge\;-C\,m_A(t).
\]
An application of Grönwall’s inequality then yields
\[
m_A(t)\;\ge\;m_A(0)\,e^{-Ct},
\quad
\forall\,t\in[0,T),
\]
and the desired bound follows from the regularity of the solution.  

To prove $ii)$, we first observe that if 
$\operatorname{supp}A_{0}=\operatorname{supp}\rho_{0}$, 
then property (i) immediately yields 
$\operatorname{supp}\rho_{0}\subset \operatorname{supp}A(t)$ for all $t$.  
Hence we assume 
$\operatorname{supp}A_{0}\subsetneq\operatorname{supp}\rho_{0}$ 
and fix any compact set 
\[
D\quad\text{with}\quad
\operatorname{supp}A_{0}\;\subset\;D\;\subset\;\operatorname{supp}\rho_{0}.
\]
We will show 
$\operatorname{supp}\rho(t)\subset\operatorname{supp}A(t)$ 
for every $t\in[0,T)$.  

Let 
\[
\Omega \;=\;\bigl\{\,x\in D : A_{0}(x)\,\rho_{0}(x)<\tfrac{\widetilde K}{2}\bigr\}.
\]
On the complement $D\setminus\Omega$, we have $A_{0}\rho_{0}\ge\tfrac{\widetilde K}{2}$, so by (i) and the uniform $L^\infty$‐bounds it follows that  
\[
A(t,x)>0
\quad\text{for all }(t,x)\in[0,T)\times\bigl(D\setminus\Omega\bigr).
\]
On $\Omega$, continuity of $A$ and $\rho$ implies there is a time $\tau>0$ such that 
$\,A(t,x)\,\rho(t,x)\le K$\,
for all $(t,x)\in[0,\tau)\times\Omega$.  Hence on $\Omega$ the equation for $A$ reads
\[
\partial_{t}A-\partial_{x}\bigl(\rho\,\partial_{x}A\bigr)\;\ge\;0
\quad\text{in }[0,\tau)\times\Omega.
\]
Since $D$  is compact and the coefficients are continuous, the parabolic minimum‐principle implies 
\[
A(t,x)>0
\quad\text{for all }(t,x)\in(0,\tau)\times\Omega.
\]
Combining this with positivity on $D\setminus\Omega$ shows 
$A(t,x)>0$ on $(0,\tau)\times D$.  Finally, applying (i) again extends this to all $t\in[0,T)$, proving 
$\operatorname{supp}\rho(t)\subset\operatorname{supp}A(t)$.  

Conversely, if $x_{0}\notin\operatorname{supp}\rho_{0}$, there exists an open interval $I_{x_0}$ such that $\rho(t,x)=0$ for all $t$ and for every $x\in I_{x_0}$. Then, from the $A$‐equation one sees
\[
\partial_{t}A(t,x)
= A(t,x)\Bigl[\alpha\,\rho-\mu\alpha(\rho-\langle\rho\rangle)
+\widetilde\beta\,A\Bigl(1-\tfrac{\rho\,A}{\widetilde K}\Bigr)\Bigr](t,x),
\]
for every $x\in I_{x_0}$
In particular, since we assume that $\text{supp}(A_0)\subset \text{supp}(\rho_0)$, we  deduce that $A(t,x)=0$ for every  $(t,x)\in[0,T)\times I_{x_0}$, so $A(t,x)=A_{0}(x)=0$ in $\in[0,T)\times I_{x_0}$.  Hence $x_{0}\notin\operatorname{supp}A(t)$.  This shows 
$\operatorname{supp}A(t)\subset\operatorname{supp}\rho_{0}$.  

Together, the two inclusions prove  
\[
\operatorname{supp}A(t)=\operatorname{supp}\rho_{0},
\quad
\forall\,t\in(0,T).
\]
\end{proof}

\bibliography{biblioart}
\bibliographystyle{abbrv}
\end{document}